%% file: main.tex
\documentclass[a4paper,UKenglish,cleveref,thm-restate]{lipics-v2021}
\usepackage{float}
\usepackage{refcount}
\usepackage{bm}
\usepackage{mathtools}
\usepackage{arydshln}

\nolinenumbers

\newcommand{\rowmajorizedby}{\prec^{\mathrm{r}}}
\newcommand{\R}{\mathbb R}
\newcommand{\Z}{\mathbb Z}
\newcommand{\veczero}{\mathbf0}
\newcommand{\vecone}{\mathbf1}
\newcommand{\matzero}{\bm O}
\renewcommand{\O}{\mathcal O}
\DeclareMathOperator{\cone}{cone}
\DeclareMathOperator{\rec}{rec}
\DeclareMathOperator{\intcone}{int.cone}
\DeclareMathOperator{\lcm}{lcm}
\newcommand{\graverbasis}{\mathcal G}

\makeatletter
\AddToHook{cmd/appendix/before}{\def\cref@section@alias{appendix}\def\cref@subsection@alias{appendix}}
\makeatother

\newcounter{theorem1}
\theoremstyle{claimstyle}
\newtheorem{claimin}{Claim}[theorem]

\newenvironment{proofof}[1]{\setcounter{theorem1}{\value{theorem}}\setcounter{theorem}{\getrefnumber{#1}}\begin{proof}[Proof of \cref{#1}]}{\end{proof}\setcounter{theorem}{\value{theorem1}}}

\crefname{lemma}{Lemma}{Lemmata}
\Crefname{lemma}{Lemma}{Lemmata}

\makeatletter
\let\ltxlabel\ltx@label 
\makeatother

\newcommand{\manualeqlabel}[1]{\refstepcounter{equation}\textup{(\theequation)}\ltxlabel{#1}}

\title{Value Functions of Separable Convex Integer Programs are Periodically Convex}

\titlerunning{Value Functions of Separable Convex Integer Programs are Periodically Convex}

\author{Koen Ligthart}{Eindhoven University of Technology, Netherlands}{k.m.ligthart@tue.nl}{https://orcid.org/0009-0004-6823-5225}{}

\authorrunning{K. Ligthart}

\Copyright{Koen Ligthart} 



\acknowledgements{The author thanks Alexandra Lassota for valuable discussions and feedback.}

\hideLIPIcs

\begin{document}

\maketitle

\begin{abstract}
\input{sections/abstract}
\end{abstract}

\input{sections/introduction}

\input{sections/introduction-periodic-convexity}

\input{sections/idp-convexity}

\input{sections/idp-bounds}

\input{sections/periodic-convexity}

\input{sections/algorithms}

\input{sections/future-directions}

\bibliography{bib}

\appendix

\input{sections/appendix-lower-bounds}

\input{sections/appendix-primal-treedepth}

\input{sections/appendix-algorithms}

\end{document}

%% file: sections/abstract.tex
We consider the periodic behavior of value functions $b\mapsto\min\{f(x)\ \vert\ Ax=b,\,x\in\Z_{\ge0}^n\}$ of integer programs. We show that there exists a positive integer $M$ depending only on the constraint matrix $A\in\Z^{m\times n}$ so that the value function is convex extensible on any subdomain of the form $r+M\Z^m$ for any $r\in\Z^m$ and any separable convex objective function $f$. With this, we extend the known periodic convexity of such functions for linear objective functions $f$, as established by Eisenbrand and Rothvoss (SODA 26), to the broader class of separable convex objective functions. We derive our main periodic convexity result by first showing that periodic convexity along lines is equivalent to the integer decomposition property of dilated polyhedra. Subsequently, we use Graver basis techniques to extend the $1$-dimensional periodic convexity to domains of arbitrary fixed dimension. We apply this periodic convexity to show that value function reformulations of block-structured integer programs become periodically convex, which yields fixed-parameter tractable (FPT) algorithms. More specifically, we optimize two-stage stochastic integer programs and $n$-fold integer programs in FPT time when parameterized by the block dimensions and coefficient size of the local blocks of the constraint matrix, allowing the coefficients of the global blocks to be large. In the setting of this parameterization, which was recently introduced by Cslovjecsek, Koutecký, Lassota, Pilipczuk, and Polak (TheoretiCS 2025), our algorithms exponentially improve on the running times of the previous optimization algorithms and extend the class of objective functions that can be optimized from linear to separable convex.

%% file: sections/introduction.tex
\section{Introduction}
\label{sec:introduction}

It is well-known that the optimal value of a convex optimization problem is a convex function of the ``right-hand sides''. That is, if $f_0,f_1,\dots,f_m\colon\R^n\to\R$ are convex functions and $S\subseteq\R^n$ is a convex set, then the \emph{value function} $h\colon\R^m\to\R\cup\{-\infty,\infty\}$ given by
\begin{equation}
    \label{eq:general-convex-value-function}
    h(b)=\inf\bigl\{f_0(x)\bigm\vert f_1(x)\le b_1,\dots,f_m(x)\le b_m,\,x\in S\bigr\}
\end{equation}
is convex. See Exercise 5.32 in~\cite{Boyd_Vandenberghe_2004}. The convexity of (\ref{eq:general-convex-value-function}) enables reformulating a \emph{two-stage stochastic} linear programming problem of the form $\min\{c_1^\top x_1+c_2^\top x_2\ \vert\ A_1x_1+A_2x_2\le d,\,x_1\in\R^{n_1},x_2\in\R^{n_2}\}$
as the convex lower-dimensional problem
\begin{equation}
    \label{eq:classical-benders}
    \min\bigl\{c_1^\top x_1+h(d-A_1x_1)\ \vert\ x_1\in\R^{n_1}\bigr\}
\end{equation}
on the \emph{first-stage variables} $x_1$ using the value function $h(b)=\min\{c_2^\top x_2\bigm\vert A_2x_2\le b,\,x_2\in\R^{n_2}\}$. Here, all vectors are column vectors. In this case, the value function $h$ is piecewise affine by LP duality and (\ref{eq:classical-benders}) can be solved by adding classical Benders cuts that represent pieces of the value function $h$~\cite{conforticornuejolszambelli}. The reformulation (\ref{eq:classical-benders}) is particularly effective when the second-stage problem $\min\{c_2^\top x_2\ \vert\ A_2x_2\le d-A_1x_1,\,x_2\in\R^{n_2}\}$ decouples into multiple independent optimization problems that can be solved in parallel, i.e., when $A_2$ is a block-diagonal matrix with diagonal blocks $D_1,\dots,D_n$ and the constraint matrix $[A_1,A_2]$ is of the form
\begin{equation}
    \label{eq:2-stage-matrix}
    \begin{bmatrix}
        C_1&D_1&&\\
        \vdots&&\ddots&\\
        C_n&&&D_n
    \end{bmatrix}.
\end{equation}
Here, the omitted parts of the matrix are zeroes.

Unfortunately, when the variable domain $S$ is discrete, the value function (\ref{eq:general-convex-value-function}) is typically not convex, which results in a value function reformulation that is difficult to solve. In this paper, we focus on the well-studied setting of integer programs, where the variables are integral and constraints are linear and have integral coefficients. Integer programs (IPs) can express a wide range of problems from combinatorial optimization and are known to be NP-hard to solve in general~\cite{conforticornuejolszambelli}. Given the relevance of IPs, it is natural ask whether the corresponding value function still possesses enough structure to derive an integer analogue to the value function reformulation of (\ref{eq:classical-benders}), which would lead to efficient algorithms to solve \emph{two-stage stochastic integer programs} of the form $\min\{f(x)\ \vert\ Ax=b,\,x\in\Z_{\ge0}^n\}$ where $A$ has the form shown in (\ref{eq:2-stage-matrix}). The theoretical behavior of value functions of integer programs been well studied~\cite{DBLP:journals/mp/BlairJ82,DBLP:journals/orl/BrownZAS21,DBLP:journals/siamjo/RomeijndersSVH16,DBLP:journals/ior/Wolsey71}, going as far back as to Gomory~\cite{gomory1969some}, who considered periodic structure that arises asymptotically. In addition, the role of value functions in practically solving stochastic integer programs has received considerable attention~\cite{DBLP:journals/mp/AhmedTS04,DBLP:journals/mp/SchultzSV98,DBLP:journals/ior/TavasliogluPS19,DBLP:journals/ior/TrappPS13}.

Only recently have Eisenbrand and Rothvoss~\cite{DBLP:conf/soda/EisenbrandR26} answered the question affirmatively in the context of exact fixed-parameter tractable (FPT) algorithms when the objective $f$ is linear. The main result of~\cite{DBLP:conf/soda/EisenbrandR26}, which is a periodically affine description of the general integer hull, implies that the value function $b\mapsto\min\{c^\top x\ \vert\ Ax\le b,\,x\in\Z^n\}$ is convex extensible on the lattice translate $r+M\Z^m=\{r+Mz\ \vert\ z\in\Z^m\}$ for any $r\in\Z^m$. Here, $M$ is an integer that only depends on the constraint matrix $A$. We refer to convex extensibility of this form as \emph{periodic convexity} with \emph{period} $M$ and refer to $r$ as the \emph{phase} vector. The periodic convexity ensures that a value function reformulation on the first-stage variables becomes a convex integer program when these variables have a fixed remainder modulo $M$. In this way, Eisenbrand and Rothvoss~\cite{DBLP:conf/soda/EisenbrandR26} derive an FPT algorithm that solves two-stage stochastic integer linear programs by guessing the remainders, i.e., the phase, of the first-stage variables and solving the corresponding value function reformulation restricted to this phase.\footnote{Eisenbrand and Rothvoss~\cite{DBLP:conf/soda/EisenbrandR26} formulate their algorithm in terms of solving a mixed-integer linear program by strengthening the integer linear programming formulation and relaxing the second-stage variables: projecting out the continuous second-stage variables yields a value function reformulation.} Their algorithm runs in time $f(k,\Delta)\cdot|I|^{\O(1)}$ when the width of the blocks in (\ref{eq:2-stage-matrix}) is bounded by $k$, the absolute value of the coefficients of $D_1,\dots,D_n$ is bounded by $\Delta$, and the encoding length of the instance is given by $|I|$.

Their algorithm generalizes the FPT algorithm for finding feasible solutions to two-stage stochastic integer programs by Cslovjecsek et al.~\cite{DBLP:journals/theoretics/CslovjecsekKLPP25}, who use a similar algorithmic approach, but only consider the periodically convex structure of the set of right-hand sides $b$ for which $Ax=b$, $x\in\Z_{\ge0}^n$ is feasible, i.e., the integer cone. The main advantage of the phase guessing algorithmic approach, which is also used in~\cite{DBLP:conf/icalp/LassotaL25}, is that it is insensitive to the sizes of coefficients in the ``global part'', $C_1,\dots,C_n$, of the constraint matrix of (\ref{eq:2-stage-matrix}). This results in less restrictive parameterizations than that of previous FPT algorithms for two-stage stochastic programs~\cite{DBLP:conf/esa/CslovjecsekEPVW21,DBLP:journals/mor/EisenbrandHKKLO25,eisenbrand2022algorithmictheoryintegerprogramming,DBLP:conf/ipco/HunkenschroderKLV25,DBLP:journals/mp/Klein22}, which rely on proximity or Graver complexity results and demand that all coefficients of the constraint matrix are bounded by the parameter $\Delta$. Cslovjecsek et al.~\cite{DBLP:journals/theoretics/CslovjecsekKLPP25} motivate the study of the less restrictive parameterization by relating it to the open parameterized complexity of the \emph{4-block integer linear programming problem}. This concerns integer linear programs with constraint matrices of the form (\ref{eq:4-block-matrix}), which simultaneously generalize two-stage stochastic constraint matrices and their transpose \emph{$n$-fold} constraint matrices, shown in~(\ref{eq:n-fold-matrix}). When parameterizing by the block dimension $k$ and maximum coefficient size $\smash{\overline\Delta}$ across the entire constraint matrix, 4-block integer programs admit a slice-wise polynomial time algorithm with a running time of the form $f(\smash{\overline\Delta})\cdot|I|^{g(k)}$~\cite{DBLP:journals/mp/HemmeckeKW14,DBLP:conf/ipco/LassotaL26,DBLP:journals/mp/OertelPW24}, but no complementing W[1]-hardness result is known and a FPT algorithm is conjectured to exist~\cite{DBLP:conf/soda/EisenbrandR26}.
\[
    \begin{array}{cc}
        \begin{bmatrix}
            B_1&\cdots&B_n\\
            D_1&&\\
            &\ddots&\\
            &&D_n
        \end{bmatrix}\quad & \quad\begin{bmatrix}
            A_0&B_1&\cdots&B_n\\
            C_1&D_1&&\\
            \vdots&&\ddots&\\
            C_n&&&D_n
        \end{bmatrix}\\
        \manualeqlabel{eq:n-fold-matrix}\quad & \quad\manualeqlabel{eq:4-block-matrix}
    \end{array}
\]
On the other hand, two-stage stochastic and $n$-fold integer programs with small coefficients, as well as their recursive generalizations, which are integer programs with constraint matrices of bounded \emph{dual treedepth} or \emph{primal treedepth}, admit FPT algorithms~\cite{DBLP:conf/soda/CslovjecsekEHRW21,DBLP:conf/esa/CslovjecsekEPVW21,DBLP:conf/icalp/EisenbrandHK18,DBLP:journals/mor/EisenbrandHKKLO25,eisenbrand2022algorithmictheoryintegerprogramming,DBLP:conf/ipco/HunkenschroderKLV25,DBLP:journals/mp/Klein22,DBLP:journals/talg/KleinR24}. The treedepth of a matrix can inductively be defined as follows: a zero matrix has a dual treedepth of $0$; and a nonzero matrix $A$ has a dual treedepth equal to the minimum $d$ so that, after permuting rows and columns, $A$ is of the form (\ref{eq:n-fold-matrix}) for $B_i\in\Z^{r\times t_i}$ and $D_i\in\Z^{s_i\times t_i}$ where all $D_i$ have dual treedepth at most $d-r$ (for $r\ne0$ or $n>1$). The primal treedepth of a matrix is the dual treedepth of its transpose.

Existing algorithms parameterized by the dual or primal treedepth $d$ and coefficient size $\smash{\overline\Delta}$ can minimize \emph{separable convex} objective functions that are accessed through a comparison oracle~\cite{DBLP:journals/mor/EisenbrandHKKLO25,eisenbrand2022algorithmictheoryintegerprogramming,DBLP:conf/ipco/HunkenschroderKLV25}, whereas the algorithms of~\cite{DBLP:journals/theoretics/CslovjecsekKLPP25,DBLP:conf/soda/EisenbrandR26} which solve two-stage stochastic and $n$-fold integer programs with large coefficients in the global parts of the constraint matrix, appear to be limited to linear objectives. Here, $f\colon\R^n\to\R\cup\{\infty\}$ is separable convex if it can be written as $f(x)=\sum_{i\in[n]}f_i(x_i)$ for convex functions $f_i\colon\R\to\R\cup\{\infty\}$, $i\in[n]:=\{1,2,\dots,n\}$. This type of objective function has received considerable attention in the integer programming literature~\cite{DBLP:conf/esa/BrandKLO24,DBLP:journals/jacm/HochbaumS90,DBLP:conf/ipco/HunkenschroderKLV25,DBLP:journals/scheduling/KnopK18}. In this paper, we study value functions of (block-structured) integer programs over such objectives. We apply our theoretical results to derive FPT algorithms for block-structured integer programs that are the first to simultaneously handle large coefficients in the global parts of the constraint matrix and separable convex objective functions.

\subsection{Contributions}

We show that value functions of integer programs with separable convex objectives possess periodic convexity. Our main result, \cref{theorem:general-convex-extensibility}, extends the consequence of the result by Eisenbrand and Rothvoss~\cite{DBLP:conf/soda/EisenbrandR26} to a more general class of objective functions.

\begin{restatable}{theorem}{theoremgeneralconvexextensibility}
    \label{theorem:general-convex-extensibility}
    Let $m,\Delta\in\Z_{\ge0}$ be given. Then there exists a positive integer $M=2^{\O((\sqrt m\Delta)^m)}$ so that the value function $b\mapsto\min\{f(x)\ \vert\ Ax=b,\,x\in\Z^n\}$ is convex extensible on $r+M\Z^m$ for any constraint matrix $A\in\{-\Delta,-\Delta+1,\dots,\Delta\}^{m\times n}$, separable convex objective function $f\colon\R^n\to\R\cup\{\infty\}$, and phase vector $r\in\Z^m$.
\end{restatable}
We note that it is known, as a result of constructions from~\cite{DBLP:journals/mp/Klein22,DBLP:journals/toct/KnopPW20}, that an $M$ satisfying \cref{theorem:general-convex-extensibility} must be at least $2^{\Omega(\Delta^m)}$. \cref{theorem:general-convex-extensibility} captures integer programs in standard form $Ax=b$, $x\in\Z_{\ge0}^n$, by adding an indicator function of the nonnegative orthant to $f$, as well as inequality form by introducing slack variables. We note that the parameterization by $m$ and $\Delta$ is different from the parameterization by $n$ and $\Delta$ in~\cite{DBLP:conf/soda/EisenbrandR26}, who consider integer linear programs in inequality form.

As a first step to show \cref{theorem:general-convex-extensibility}, we establish that periodic convexity along $1$-dimensional domains is equivalent to the integer decomposition property (IDP) of dilated polyhedra defined by the constraint matrix $A$. Polyhedra that have the IDP, also called normal or integrally closed, and their connections to integer programming are well-studied~\cite{10.1007/978-3-642-95322-4_2,DBLP:journals/combinatorics/CoxHHH14,cox2024toric,DBLP:journals/dmtcs/KolmanKT20,DBLP:books/daglib/0090562}. Despite this, the simple connection to periodic convexity appears to not have been mentioned in the literature. After establishing the equivalence, we apply Graver basis techniques to extend the convexity to higher-dimensional domains as in \cref{theorem:general-convex-extensibility}. Additionally, we provide parameterized bounds on the needed dilation to establish the IDP for polyhedra defined by block-structured constraint matrices, which yields a restricted variant of \cref{theorem:general-convex-extensibility} with an improved bound on $M$ when $A$ is block-structured.

These periodic convexity results are exploited to derive simple FPT algorithms for block-structured integer programs. In particular, we treat both two-stage stochastic integer programs in \cref{theorem:2-stage-algorithm}, as well as $n$-fold integer programs in \cref{theorem:n-fold-algorithm}. These algorithms are the first to simultaneously support large entries in the global parts of the constraint matrix as well as (nonlinear) separable convex objective functions, and affirmatively answer a question raised by Koutecký~\cite{DBLP:conf/iwpec/Koutecky25}.

\begin{restatable}{theorem}{theoremtwostagealgorithm}
    \label{theorem:2-stage-algorithm}
    Let $C_i\in\Z^{t\times r}$, $D_i\in\{-\Delta,-\Delta+1,\dots,\Delta\}^{t\times s}$, $l_0,u_0\in\Z^r$, $l_i,u_i\in\Z^s$, and $b_i\in\Z^t$ for $i\in[n]$. Let $f\colon\R^{r+sn}\to\R$ be separable convex and accessible through a comparison oracle on $\Z^{r+sn}$. An optimal solution to the two-stage stochastic integer program
    \begin{align*}
        \min\bigl\{f(x_0,x_1,\dots,x_n)\bigm\vert C_ix_0+D_ix_i=b_i,\,{}&l_i\le x_i\le u_i,\,x_i\in\Z^s,\,i\in[n],\\
        &l_0\le x_0\le u_0,\,x_0\in\Z^r\bigr\}
    \end{align*}
    can be found in time $2^{\O((\sqrt t\Delta)^t\cdot r+r^2\log r)}\cdot s\log^{\O(1)}(s)\cdot n\cdot L^{\O(1)}$, where $L=\log\max_{i\in\{0,1,\dots,n\}}\linebreak\|u_i-l_i\|_\infty$.
\end{restatable}

\begin{restatable}{theorem}{theoremnfoldalgorithm}
    \label{theorem:n-fold-algorithm}
    Let $B\in\Z^{r\times t}$, $D_i\in\{-\Delta,-\Delta+1,\dots,\Delta\}^{s\times t}$, $b_0\in\Z^r$, $l_i,u_i\in\Z^t$, and $b_i\in\Z^s$ for $i\in[n]$. Let $f\colon\R^{tn}\to\R$ be separable convex and accessible through a comparison oracle on $\Z^{tn}$. An optimal solution to the $n$-fold integer program
    \begin{equation}
        \label{eq:n-fold-ip}
        \begin{aligned}
            \min\bigl\{f(x_1,\dots,x_n)\bigm\vert{}&Bx_1+\dots+Bx_n=b_0,\\
            &D_ix_i=b_i,\,l_i\le x_i\le u_i,\,x_i\in\Z^t,\,i\in[n]\bigr\}
        \end{aligned}
    \end{equation}
    can be found in time $2^{\O((t+s)\Delta)^{t(t+s)}}\cdot r\cdot n\log^{\O(1)}(n)\cdot L^{\O(1)}$, where $L=\log\max_{i\in[n]}\|u_i-l_i\|_\infty$.
\end{restatable}
The running times of our algorithms measure the number of arithmetic operations and comparison oracle calls performed, as is standard in the literature.

For both \cref{theorem:2-stage-algorithm,theorem:n-fold-algorithm}, we employ value function reformulations to obtain equivalent periodically convex integer programs on a parameterized number of variables. For a fixed phase of the variables, the objective functions become convex extensible and the corresponding problems can be solved using the algorithm by Veselov et al.~\cite{DBLP:journals/dam/VeselovGZC20}. To evaluate the objective function, we solve block-structured IPs with small coefficients with the algorithms from~\cite{DBLP:conf/ipco/HunkenschroderKLV25}.

The parametric running time dependencies of \cref{theorem:2-stage-algorithm,theorem:n-fold-algorithm} are doubly exponential. This improves upon the previous algorithms that support large entries and linear objectives in two-stage stochastic IPs~\cite{DBLP:conf/soda/EisenbrandR26} and $n$-fold IPs~\cite{DBLP:journals/theoretics/CslovjecsekKLPP25}, which have triply exponential parametric dependencies in the regime of square block dimensions. The doubly exponential dependency of the two-stage stochastic IP algorithm matches the previous feasibility algorithm by Cslovjecsek et al.~\cite{DBLP:journals/theoretics/CslovjecsekKLPP25} and approaches the known $2^{2^{\delta(s+t)}}$ running time lower bound ($\delta>0$) for two-stage stochastic IPs with $r=1$ under the ETH~\cite{DBLP:journals/mp/JansenKL23}. We also note that previous proximity based two-stage stochastic IP algorithms~\cite{DBLP:conf/esa/CslovjecsekEPVW21,DBLP:conf/ipco/HunkenschroderKLV25} have a doubly exponential runtime dependence on the number $r$ of first-stage variables, whereas the algorithms by Eisenbrand and Rothvoss~\cite{DBLP:conf/soda/EisenbrandR26} and Cslovjecsek et al.~\cite{DBLP:journals/theoretics/CslovjecsekKLPP25}, and the algorithm from \cref{theorem:2-stage-algorithm} do not have this.

Finally, we apply the same value function reformulation framework as the one underlying \cref{theorem:2-stage-algorithm,theorem:n-fold-algorithm} with the following additional results: we extend the result from \cref{theorem:n-fold-algorithm} to IPs with bounded dual treedepth; derive an efficient algorithm for $4$-block integer programs with a single global variable, which has applications discussed by Chen, Chen, and Zhang~\cite{DBLP:journals/mp/ChenCZ24}; and partially answer a question by Brand, Kouteck\'y, and Ordyniak~\cite{DBLP:conf/aaai/BrandKO21} about the complexity of structured mixed-integer linear programming problems.

%% file: sections/introduction-periodic-convexity.tex
\section{Periodic Convexity of Value Functions of Separable Convex IPs}
\label{sec:intro-periodic-convexity}

In this section, we show the convex extensibility of value functions of the form $b\mapsto\linebreak\min\{f(x)\ \vert\ Ax=b,\,x\in\Z^n\}$ on lattice translates $r+M\Z^m$ for any $r\in\Z^m$ and some large number $M$ depending only on $A$. Recall that a function $g\colon D\to\R\cup\{\infty\}$ is convex extensible on a set $S\subseteq D\subseteq\R^m$ if there exists a convex function $\tilde g\colon\R^m\to\R\cup\{\infty\}$ so that $g(y)=\tilde g(y)$ for all $y\in S$. If the subdomain $S$ is not specified, we consider convex extensibility on the domain $D$ of the function. Thus, we wish to show that the function $h\colon\Z^n\to\R\cup\{-\infty,\infty\}$ given by $h(z)=\min\{f(x)\ \vert\ Ax=r+Mz,\,x\in\Z^n\}$ is convex extensible for any $r\in\Z^m$. Here, we use the convention that infeasible problems have an objective value of $\infty$ and unbounded problems have an objective value of $-\infty$. In the proof of \cref{lemma:fixed-dimension-periodic-convexity}, we make the straightforward observation that it is sufficient to show that
\begin{equation}
    \label{eq:0-centered-convexity}
    h(\veczero)\le\tfrac1k\sum_{i\in[k]}h(z^i)
\end{equation}
when $z^1,\dots,z^k\in\Z^n$ are integral points that sum to the zero vector $\veczero$ and all $h(z^i)$ are finite. To derive (\ref{eq:0-centered-convexity}), we let $x^i$ be a solution attaining the minimum value $f(x^i)=h(z^i)$ for each $i\in[k]$ and modify these solutions to obtain new solutions $\hat x^i$ to the systems $A\hat x^i=r+M\hat z^i$ without increasing the total sum of the function values, i.e., $f(\hat x^1)+\dots+f(\hat x^k)\le f(x^1)+\dots+f(x^k)$. Multiple of such modifications are applied in a pairwise fashion, eventually resulting in $\hat z^i=\veczero$. In this way, the best resulting solution to $Ax=r$ will be a witness for (\ref{eq:0-centered-convexity}).

To accurately capture the preservation of objective values, we relate the intermediary solution sets through the following well-known vector partial order: a vector $u\in\R^k$ is \emph{majorized} by $v\in\R^k$, denoted by $u\prec v$, if $\sum_{i\in[k]}g(u_i)\le\sum_{i\in[k]}g(v_i)$ for all univariate convex functions $f\colon\R\to\R\cup\{\infty\}$. See~\cite{marshall1979inequalities} for a comprehensive overview of this partial order. To capture separable convex functions in an $n$-dimensional space, we consider row-wise majorization of matrices. We use $[x^1,\dots,x^k]$ to denote horizontal and $[x^1;\dots;x^k]$ to denote vertical concatenation respectively. Thus, if $x^1,\dots,x^k$ are $n$-dimensional vectors, $[x^1,\dots,x^k]$ denotes the $n\times k$ matrix with columns $x^1,\dots,x^k$. In this way, $f(\hat x^1)+\dots+f(\hat x^k)\le f(x^1)+\dots+f(x^k)$ holds universally for all separable convex $f\colon\R^n\to\R\cup\{\infty\}$ if and only if each row of $[\hat x^1,\dots,\hat x^k]$ is majorized by the corresponding row of $[x^1,\dots,x^k]$. We denote this by $[\hat x^1,\dots,\hat x^k]\rowmajorizedby[x^1,\dots,x^k]$.

Let $x^1,x^2\in\R^n$ and let $y\in\R^n$ describe a partial solution exchange $\hat x^1=x^1+y$, $\hat x^2=x^2-y$. Such modification $y$ preserves the objective value, i.e., $[\hat x^1,\hat x^2]\rowmajorizedby[x^1,x^2]$, if and only if the modification vector $y$ is \emph{conformal} to the difference $x^2-x^1$~\cite{marshall1979inequalities}. Here, a vector $y\in\R^n$ is conformal to a vector $u\in\R^n$, denoted by $y\sqsubseteq u$, if $y_iu_i\ge0$ and $|y_i|\le|u_i|$ for all $i\in[n]$. A way to obtain integral conformal exchange vectors is by using properties of the \emph{Graver basis} $\graverbasis(A)\subseteq\Z^n$ of a matrix $A\in\Z^{m\times n}$, which is the finite set of $\sqsubseteq$-minimal nonzero integral kernel elements of $A$. A fundamental property of the Graver basis is that any $x\in\Z^n$ in the kernel of $A$ can be written as the sum of conformal Graver basis elements, i.e., $x=g^1+\dots+g^\ell$ for $g^i\in\graverbasis(A)$, $g^i\sqsubseteq x$ for $i\in[\ell]$. The \emph{$\ell_p$-Graver complexity} of a matrix $A$ is defined as $g_p(A)=\max_{g\in\graverbasis(A)}\|g\|_p$. It is known that the $\ell_1$-Graver complexity of a matrix $A\in\{-\Delta,-\Delta+1,\dots,\Delta\}^{m\times n}$ is bounded from above by $g_1(A)\le(2m\Delta+1)^m$~\cite{DBLP:conf/icalp/EisenbrandHK18}.

In \cref{sec:idp-convexity} and \cref{sec:periodic-convexity}, we will provide the needed pairwise operations to establish (\ref{eq:0-centered-convexity}). \cref{sec:idp-bounds} bounds the value of $M$ that enables the operation from \cref{sec:idp-convexity}.

%% file: sections/idp-convexity.tex
\subsection{The integer decomposition property and midpoint convexity}
\label{sec:idp-convexity}

Here, we will provide the first needed pairwise modification of solutions by showing that there is a positive integer $M$ so that if $x^1,x^2\in\Z^n$ and $r,z\in\Z^m$ are such that $Ax^1=r+M\cdot(-z)$ and $Ax^2=r+Mz$, then there exist $\hat x^1,\hat x^2\in\Z^n$ so that $A\hat x^1=A\hat x^2=r$ and $[\hat x^1,\hat x^2]\rowmajorizedby[x^1,x^2]$. This coincides with periodic midpoint convexity of the corresponding value function. We connect this to the integer decomposition property of dilated polyhedra.

A polyhedron $P$ is said to have the \emph{integer decomposition property}~\cite{10.1007/978-3-642-95322-4_2} (IDP) if for any $k\in\Z_{\ge0}$, any integral point $x\in kP\cap\Z^n$ in the $k$-dilation of $P$ can be written as the sum of $k$ integral points in the original polyhedron $P$, i.e., there exist $x^1,\dots,x^k\in P$ so that $x=x^1+\dots+x^k$. It is well-known that a rational polyhedron that has the IDP must be integral and that a $z$-dilation of an $n$-dimensional integral polyhedron has the IDP for any positive integer $z\ge n-1$. See Theorem 2.2.12 in~\cite{cox2024toric}. Therefore, there exists a dilation for any rational polyhedron so that the resulting dilated polyhedron has the IDP. We return to computing dilation factors that establish the IDP in \cref{sec:idp-bounds}.

Having the interest of parameterized algorithms in mind, we describe properties of parameterized classes of matrices and the polyhedra that are defined by such matrices as done in \cref{def:idp}. Here, a class of constraint matrices is a set of matrices with integral coefficients and some varying dimensions.
\begin{definition}
    \label{def:idp}
    A class $\mathcal A$ of constraint matrices has the integer decomposition property after an $M$-dilation if $M$ is a positive integer and the polyhedron $M\cdot\{x\in\R_{\ge0}^n:Ax=b\}=\{x\in\R_{\ge0}^n:Ax=Mb\}$ has the IDP for any $m\times n$ matrix $A\in\mathcal A$ and integral vector $b\in\Z^m$.
\end{definition}
It deserves to be noted that the class of matrices $A$ for which the undilated polyhedron $\{x\in\R_{\ge0}^n:Ax\le b\}$ has the IDP for all $b\in\Z^m$ is precisely the class of totally unimodular matrices~\cite{10.1007/978-3-642-95322-4_2}. Thus, the class of TU matrices has the IDP after a $1$-dilation.

The classes $\mathcal A$ of constraint matrices that we consider will be closed under inverting columns, i.e., $A\in\mathcal A$ implies that any matrix $A'$ arising from $A$ by inverting the sign of the coefficients in a column also satisfies $A'\in\mathcal A$. In this setting, we can make \cref{observation:idp-closed-under-inversion} by inverting the sign of the appropriate columns in the definition of the IDP.

\begin{observation}
    \label{observation:idp-closed-under-inversion}
    Let $\mathcal A$ be a class of constraint matrices that is closed under inverting columns and that has the IDP after an $M$-dilation. Let $A\in\mathcal A$, $A\in\Z^{m\times n}$, $x\in\Z^n$, $k\in\Z_{\ge0}$, and $b\in\Z^m$ be such that $Ax=kMb$. Then there exists a decomposition $x=x^1+\dots+x^k$ with $x^i\in\Z^n$, $x^i\sqsubseteq x$, and $Ax^i=Mb$ for all $i\in[k]$.
\end{observation}
This simple observation paves the way for connecting the IDP to the periodic convex extensibility of integer program value functions along \emph{lines}, being $1$-dimensional affine subspaces. We first consider the special case of midpoint convexity in \cref{lemma:idp-implies-midpoint-convexity}, which yields the pairwise solution modification operation we will need to prove \cref{theorem:general-convex-extensibility}.

\begin{lemma}
    \label{lemma:idp-implies-midpoint-convexity}
    Let $\mathcal A$ be a class of constraint matrices that is closed under inverting columns and has the IDP after an $M$-dilation. Let $A\in\mathcal A$, $A\in\Z^{m\times n}$ and $r\in\Z^m$ be given. Let $x^A,x^B\in\Z^n$ and $z\in\Z^m$ be such that $Ax^A=r+M\cdot-z$ and $Ax^B=r+M\cdot z$. Then there exist $\hat x^A,\hat x^B\in\Z^n$ so that $A\hat x^A=A\hat x^B=r$ and $[\hat x^A,\hat x^B]\rowmajorizedby[x^A,x^B]$.
\end{lemma}

\begin{proof}
    Let $y=x^B-x^A$, which satisfies $Ay=2Mz$. We use \cref{observation:idp-closed-under-inversion} to find a decomposition $y=y^1+y^2$ with $y^i\in\Z^n$, $y^i\sqsubseteq y=x^B-x^A$, and $Ay^i=Mz$ for all $i\in[2]$. Let $\hat x^A=x^A+y^1$ and $\hat x^B=x^B-y^1=x^A+y^2$, which immediately implies that $[\hat x^A,\hat x^B]\rowmajorizedby[x^A,x^B]$. Furthermore, we have that $A\hat x^A=Ax^A+Ay^1=r+M\cdot-z+M\cdot z=r$. The condition for $\hat x^B$ is verified analogously.
\end{proof}

The relation between dilations establishing the IDP and periodic convex extensibility is more precisely captured in \cref{proposition:idp-iff-convex-extensibility-along-line}, which we believe to be of independent interest.

\begin{proposition}
    \label{proposition:idp-iff-convex-extensibility-along-line}
    Let $\mathcal A$ be a class of constraint matrices that is closed under inverting columns. Then $\mathcal A$ has the IDP after an $M$-dilation if and only if the value function $h\colon\Z^m\to\R\cup\{\infty\}$ given by $h(z)=\min\{f(x)\ \vert\ Ax=r+Mz,\,x\in\Z^n\}$ is convex extensible on any line for any separable convex $f\colon\R^n\to\R\cup\{\infty\}$ and offset vector $r\in\Z^m$.
\end{proposition}

\begin{proof}
    The necessity of the IDP follows from an inductive argument on $k$ in the definition of the IDP. The base case $k=0$ is trivial. For the induction step, let $x\in\Z_{\ge0}^n$ and $b\in\Z^m$ be so that $Ax=kMb$. Define the separable convex indicator function $f$ by setting $f(y)=0$ if $\veczero\le y\le x$ and $f(y)=\infty$ otherwise. Now, $h(\veczero)=0$ because $A\cdot\veczero=M\cdot\veczero$. Combining this with $h(kb)=0$ and the assumed convex extensibility of $h$ along the line through $\veczero$ and $kb$ shows that $h(b)\le\tfrac{k-1}kh(\veczero)+\tfrac1kh(kb)=0$ and thus that there exists an $x^*\in\Z^n$ satisfying $\veczero\le x^*\le x$ and $Ax^*=Mb$. Now $A(x-x^*)=(k-1)Mb$ and the induction hypothesis complete the argument.

    With regards to the convex extensibility along any line, we need to show that for any $z^A,z^B\in\Z^m$ and $\lambda\in[0,1]$ so that $z^*=\lambda z^A+(1-\lambda)z^B\in\Z^m$, it holds that $h(z^*)\le\lambda h(z^A)+(1-\lambda)h(z^B)$. For $z^*$ to be integral, it must hold that $\lambda(z^A-z^B)\in\Z^m$. Therefore, by excluding the trivial cases where $z^A=z^B$ or $\lambda\in\{0,1\}$, we find that $\lambda$ is equal to a rational number $\tfrac{k-p}k$ for coprime integers $p$ and $k$ satisfying $0<p<k$. The fact that $p$ and $k$ are coprime implies that $b:=\tfrac1k(z^B-z^A)\in\Z^m$ is integral. If $h(z^A)=\infty$ or $h(z^B)=\infty$ we are done. Furthermore, we may assume that $h(z^A),h(z^B)>-\infty$ because the unbounded case reduces to the bounded case as argued at the end of the proof of \cref{lemma:fixed-dimension-periodic-convexity}. Let $x^A$ and $x^B$ be integral solutions to $Ax^A=r+Mz^A$ and $Ax^B=r+Mz^B$ that attain $f(x^A)=h(z^A)$ and $f(x^B)=h(z^B)$. Let $y=x^B-x^A$, which satisfies $Ay=M(z^B-z^A)=kMb$. Therefore, we can apply \cref{observation:idp-closed-under-inversion} to obtain $y^i\in\Z^n$, $y^i\sqsubseteq y$, and $Ay^i=Mb$ for $i\in[k]$ so that $y=y^1+\dots+y^k$. Write $f(x)=\sum_{j\in[n]}f_j(x_j)$ for univariate convex $f_j\colon\R\to\R\cup\{\infty\}$. Let $j\in[n]$ be so that $y_j\ne0$ and let $S\subseteq[k]$. Convexity combined with $y_j^iy_j\ge0$ for all $i\in[k]$ shows that
    \begin{equation}
        \label{eq:univariate-convex-steps}
        f_j\bigl(x_j^A+\sum_{i\in S}y_j^i\bigr)\le\bigl(1-\tfrac1{y_j}\sum_{i\in S}y_j^i\bigr)\cdot f_j(x_j^A)+\tfrac1{y_j}\sum_{i\in S}y_j^i\cdot f_j(x_j^B).
    \end{equation}
    Since every $i\in[k]$ appears in exactly $\binom{k-1}{p-1}=\tfrac pk\binom kp$ cardinality $p$ subsets of $[k]$, summing (\ref{eq:univariate-convex-steps}) over all cardinality $p$ subsets $S\in\binom{[k]}p$ yields
    \begin{align*}
        &\sum_{S\in\binom{[k]}p}f_j\bigl(x_j^A+\sum_{i\in S}y_j^i\bigr)\\
        &\quad\quad\le\Bigl(\tbinom kp-\tfrac1{y_j}(\tfrac pk\tbinom kpy_j^1+\dots+\tfrac pk\tbinom kpy_j^k)\Bigr)\cdot f_j(x_j^A)\\
        &\quad\quad\quad+\tfrac1{y_j}\Bigl(\tfrac pk\tbinom kpy_j^1+\dots+\tfrac pk\tbinom kpy_j^k\Bigr)\cdot f_j(x_j^B)\\
        &\quad\quad=\tbinom kp\bigl((1-\tfrac pk)\cdot f_j(x_j^A)+\tfrac pk\cdot f_j(x_j^B)\bigr)\\
        &\quad\quad=\tbinom kp\bigl(\lambda f_j(x_j^A)+(1-\lambda)f_j(x_j^B)\bigr),
    \end{align*}
    of which the outer inequality also holds if $y_j=0$. Therefore, summing over all $j\in[n]$ gives
    \[
        \sum_{S\in\binom{[k]}p}f\bigl(x^A+\sum_{i\in S}y^i\bigr)\le\tbinom kp\bigl(\lambda f(x^A)+(1-\lambda)f(x^B)\bigr)
    \]
    An averaging argument shows that there exists a cardinality $p$ subset $S\subseteq[k]$ for which
    \[
        f\bigl(x^A+\sum_{i\in S}y^i\bigr)\le\lambda f(x^A)+(1-\lambda)f(x^B).
    \]
    Now, $x^*=x_j^A+\sum_{i\in S}y_j^i$ is an integral solution to the system
    \begin{align*}
        Ax^*&=Ax^A+\sum_{i\in S}Ay^i=r+Mz^A+|S|\cdot Mb=r+Mz^A+\tfrac pk\cdot M(z^B-z^A)\\
        &=r+M\bigl(\lambda z^A+(1-\lambda)z^B\bigr)=r+Mz^*
    \end{align*}
    and we can conclude that
    \[
        h(z^*)\le f(x^*)\le\lambda f(x^A)+(1-\lambda)f(x^B)=\lambda h(z^A)+(1-\lambda)h(z^B).\qedhere
    \]
\end{proof}

%% file: sections/idp-bounds.tex
\subsection{Dilations that establish the integer decomposition property}
\label{sec:idp-bounds}

To make use out of \cref{lemma:idp-implies-midpoint-convexity}, we will now consider the values of $M$ that ensure that classes of constraint matrices have the IDP after an $M$-dilation. We have already hinted at how to obtain such dilation for general constraint matrices $A\in\{-\Delta,-\Delta+1,\dots,\Delta\}^{m\times n}$. A combination of the Hadamard bound and Cramer's rule shows that each vertex of $\{x\in\R_{\ge0}^n:Ax=b\}$ has rational coordinates with denominators that are bounded by $(\sqrt m\Delta)^m$. Therefore, dilating such polyhedron by a factor of $\lcm\{1,2,\dots,\lfloor(\sqrt m\Delta)^m\rfloor\}$ yields an integral polyhedron. Here, we may bound $\lcm\{1,2,\dots,N\}=2^{\Theta(N)}$~\cite{Nair01021982}. Now an $n$-dilation of this integral polyhedron, yields the IDP~\cite{cox2024toric}. To obtain an $n$-independent bound, we exploit the fact that $A$ has at most $(2\Delta+1)^m$ distinct columns. It then suffices to show that the IDP of the dilated bounded-dimension polyhedron defined by the matrix consisting of exactly these $(2\Delta+1)^m$ possible columns implies the IDP of $A$ after the right dilations. For the sake of completeness, we explicitly provide this straightforward argument in \cref{claim:standard-matrix}, where we use the same construction as in~\cite{DBLP:conf/ipco/LassotaL26}.

\begin{lemma}
    \label{lemma:general-scaled-idp}
    Let $m,\Delta\in\Z_{\ge0}$ be given. Then there exists a positive integer $M=2^{\O((\sqrt m\Delta)^m)}$ so that the class of matrices $A\in\{-\Delta,-\Delta+1,\dots,\Delta\}^{m\times n}$ has the IDP after an $M$-dilation.
\end{lemma}

\begin{proof}
    \cref{claim:standard-matrix} shows that it suffices to consider a single matrix with a bounded number of columns.

    \begin{claimin}
        \label{claim:standard-matrix}
        Let $A\in\Z^{m\times n}$, $b\in\Z^m$ and let $\hat A\in\Z^{m\times\smash{\hat n}}$ be a matrix obtained by adding a column to $A$ or removing a duplicate column from $A$. If $\{\hat x\in\R_{\ge0}^{\smash{\hat n}}:\hat A\hat x=Mb\}$ has the IDP, then $\{x\in\R_{\ge0}^n:Ax=Mb\}$ has the IDP.
    \end{claimin}

    \begin{claimproof}
        Let $k\in\Z_{\ge0}$ be arbitrary. First, consider the case where a matrix $\hat A\in\Z^{m\times(n+1)}$ arises from a matrix $A\in\Z^{m\times n}$ by adding an $(n+1)$-th column to $A$. Let $x\in\Z_{\ge0}^n$ satisfy $Ax=kMb$. Then $\hat x$ obtained by padding $x$ with a zero in the $(n+1)$-th coordinate is a solution to $\hat A\hat x=kMb$. We obtain the decomposition $\hat x=\hat x^1+\dots+\hat x^k$ where $\hat A\hat x^i=Mb$ and $\hat x^i\in\Z_{\ge0}^{n+1}$ for $i\in[k]$. By letting $x^i\in\Z_{\ge0}^n$ be the projection of $\hat x^i$ on the first $n$ coordinates, we obtain the required decomposition $x=x^1+\dots+x^k$ with $Ax^i=Mb$.

        Second, consider the case where $A\in\Z^{m\times(n+1)}$ has identical $n$-th and $(n+1)$-th columns and $\hat A\in\Z^{m\times n}$ consists of the first $n$ columns of $A$. In this case, a solution $x\in\Z_{\ge0}^{n+1}$ to $Ax=kMb$ can be aggregated into a solution $\hat x\in\Z_{\ge0}^n$ to $\hat A\hat x=b$ by setting $\hat x_j=x_j$ for $j\in[n-1]$ and $\hat x_n=x_n+x_{n+1}$. We obtain the decomposition $\hat x=\hat x^1+\dots+\hat x^k$ where $\hat A\hat x^i=Mb$ and $\hat x^i\in\Z_{\ge0}^n$ for $i\in[k]$. Reverse the aggregation operation by copying $x_j^i=\hat x_j^i$ for $j\in[n-1]$ and distributing the last component $\hat x_n^i$ over $x_n^i$ and $x_{n+1}^i$ through
        \[
            [x^i_n;x^i_{n+1}]=
            \begin{cases}
                [\hat x_n^i;0],&\text{if $\sum_{j=1}^i\hat x_n^j\le x_n$,}\\[10pt]
                [x_n-\sum_{j=1}^{i-1}\hat x_n^j;\ \sum_{j=1}^i\hat x_n^j-x_n],&\text{if $\sum_{j=1}^{i-1}\hat x_n^j\le x_n<\sum_{j=1}^i\hat x_n^j$,}\\[10pt]
                [0;\hat x_n^i],&\text{if $x_n<\sum_{j=1}^{i-1}\hat x_n^j$.}
            \end{cases}
        \]
        By construction, the vectors $x^i\in\Z_{\ge0}^{n+1}$ are solutions to $Ax^i=Mb$. Therefore, we obtain the decomposition $x=x^1+\dots+x^k$.
    \end{claimproof}

    By applying \cref{claim:standard-matrix}, it suffices to prove \cref{lemma:general-scaled-idp} for the matrix $A$ with column set $\{-\Delta,-\Delta+1,\dots,\Delta\}^m$ which has $n=(2\Delta+1)^m$ columns. As argued earlier, for $N=\lcm\{1,2,\dots,\lfloor(\sqrt m\Delta)^m\rfloor\}\le2^{\O((\sqrt m\Delta)^m)}$, the dilated polyhedron $P=N\cdot\{x\in\R_{\ge0}^n:Ax=b\}$ is integral for any $b\in\Z^m$. Now an additional dilation with a factor of $n$ yields the IDP~\cite{cox2024toric}. Therefore, $M=n\cdot N\le(2\Delta+1)^m\cdot2^{\O((\sqrt m\Delta)^m)}$ suffices.
\end{proof}

For our application to integer programs with constraint matrices of bounded dual treedepth, we will need a bound on the dilation $M$ that establishes the IDP that is independent of the total number of constraints $m$. We note that it is known that many block-structured constraint matrices yield polyhedra with vertices of bounded fractionality~\cite{DBLP:conf/aaai/BrandKO21}. Unfortunately, stronger properties are needed as not all integral polyhedra have the IDP~\cite{cox2024toric}. An important property of matrices $A\in\{-\Delta,-\Delta+1,\dots,\Delta\}^{m\times n}$ with dual treedepth at most $d$ is that their $\ell_1$-Graver complexity is at most $\O(\Delta)^{2^d}$~\cite{DBLP:journals/toct/KnopPW20}. In this section, we will prove \cref{lemma:idp-dual-glueing}, which shows that this Graver complexity bound implies a parameterized bound on the multiplier $M$ needed to establish the IDP for matrices with bounded dual treedepth (\cref{corollary:dual-treedepth-scaled-idp}). As a special case, we obtain a bound for $n$-fold matrices (\cref{corollary:n-fold-scaled-idp}).

\begin{lemma}
    \label{lemma:idp-dual-glueing}
    Let $N$, $U$, and $r$ be positive integers. Then there exists a positive integer $M=2^{\O(rU)^r}\cdot N$ so that the following holds: let $\mathcal D$ be a class of matrices that has the IDP after an $N$-dilation. Let $\mathcal A$ be the class of matrices of the form (\ref{eq:n-fold-matrix}) for $B_i\in\Z^{r\times t_i}$ and $D_i\in\mathcal D$, $D\in\Z^{s_i\times t_i}$ that additionally satisfy
    \[
        g_\infty\left(\begin{bmatrix}B_i&I\\D_i&\end{bmatrix}\right)\le U
    \]
    for all $i\in[n]$. Then $\mathcal A$ has the IDP after an $M$-dilation.
\end{lemma}

We show \cref{lemma:idp-dual-glueing} by independently applying the IDP on systems of the form $D_ix_i=kMb_i$ to obtain a large number of subsolutions for right-hand side vectors $Nb_i$. For each such block, we can ``balance'' the subsolutions with respect to the constraints $B_i$ so that their contribution with respect to the global constraints becomes close to average. Then, similarly as done for the Graver complexity bound by Eisenbrand, Hunkenschröder, and Klein~\cite{DBLP:conf/icalp/EisenbrandHK18}, we use a variant of the Steinitz lemma~\cite{grinbergsevastyanovsteinitz} to guarantee that there are collections of a bounded number of subsolutions that attain exactly the average global constraint contribution. Finally, we apply \cref{lemma:general-scaled-idp} for $d=1$ to show that we can merge these bounded collections into collections of exactly $2^{\O(rU)^r}$ subsolutions each.

We will need the colorful version of the Steinitz lemma from~\cite{DBLP:journals/mp/OertelPW24}, which is formulated in \cref{lemma:colorful-steinitz-lemma} with the improved bound from~\cite{Barany+2024+261+267}.

\begin{lemma}[Colorful Steinitz Lemma~\cite{Barany+2024+261+267,DBLP:journals/mp/OertelPW24}]
    \label{lemma:colorful-steinitz-lemma}
    Let $\|\cdot\|$ be a norm on $\R^d$. Let $(x_j^i)_{j\in[n]}^{i\in[m]}$ be a matrix of vectors with norm at most $L$ so that $\sum_{i\in[m]}\sum_{j\in[n]}x_j^i=\veczero$. Then there exist permutations $\pi_1,\dots,\pi_n\colon[m]\to[m]$ so that
    \[
        \Bigl\|\sum_{i\in[k]}\sum_{j\in[n]}x_j^{\pi_j(i)}\Bigr\|\le(4d-2)L
    \]
    for all $k\in[m]$.
\end{lemma}

Additionally, we need \cref{lemma:balancing} to ensure that the vectors in the application of \cref{lemma:colorful-steinitz-lemma} are bounded. \cref{lemma:balancing} is similar in spirit to an intermediary result in Lemma 10 in~\cite{DBLP:journals/mp/KnopKLMO23}, but provides a stronger error bound only depending on $g_\infty(A)$, independent of $n$.

\begin{lemma}
    \label{lemma:balancing}
    Let $A\in\Z^{m\times n}$, $b\in\Z^m$, and let $x^1,\dots,x^k\in\Z^n$, $k>0$, be so that $Ax^i=b$ for all $i\in[n]$. Denote $x^*=\tfrac1k(x^1+\dots+x^k)$. Then, there exist $\hat x^1,\dots,\hat x^k\in\Z^n$ so that $A\hat x^i=b$ and $\|\hat x^i-x^*\|_\infty\le g_\infty(A)$ for all $i\in[n]$, and $[\hat x^1,\dots,\hat x^k]\rowmajorizedby[x^1,\dots,x^k]$.
\end{lemma}

\begin{proof}
    Define the error vectors $e^i\in\R^n$ by $e^i=x^i-x^*$ for $i\in[n]$. We prove the lemma by induction on the potential $\sum_{i\in[k]}\|e^i\|_1$. Note that if the potential is at most $g_\infty(A)$, the trivial choice $\hat x^i=x^i$ satisfies the conditions of the lemma. Hence, we may focus on the induction step. If there is no $j^*\in[n]$ and $i\in[k]$ with $|e_{j^*}^i|>g_\infty(A)$, we are done. Otherwise, w.l.o.g.,\ assume that $e_{j^*}^1<-g_\infty(A)$ (the other sign is treated analogously). Since $\sum_{i\in[k]}e_{j^*}^i=0$, there must exist another $i'\in[k]$ with $e_{j^*}^{i'}>0$. W.l.o.g., this is $i'=2$. Note that $x^2-x^1$ is in the integer kernel of $A$ and can therefore be decomposed into a sum $x^2-x^1=g^1+\dots+g^\ell$ of conformal Graver basis elements $g^l\in\graverbasis(A)$, $g^l\sqsubseteq x^2-x^1$, $l\in[\ell]$. Choose a Graver basis element $g^l$, $l\in[\ell]$ for which $g_{j^*}^l>0$ and set $\hat x^1=x^1+g^l$, $\hat x^2=x^2-g^l$. Now consider the collection $\hat x^1,\hat x^2,x^3,\dots,x^k$. By conformality of $g^l\sqsubseteq x^2-x^1$, we have that $[\hat x^1,\hat x^2,x^3,\dots,x^k]\rowmajorizedby[x^1,\dots,x^k]$. Additionally, note that the solutions remain solutions to $Ax=b$ since $Ag^l=0$, and that the average $x^*$ is preserved. We now inspect the change in the potential, which is equal to
    \[
        \|\hat x^1-x^*\|_1+\|\hat x^2-x^*\|_1-\|e^1\|_1-\|e^2\|_1=\|e^1+g^l\|_1+\|e^2-g^l\|_1-\|e^1\|_1-\|e^2\|_1
    \]
    The contribution of the $j$-th coordinate to the potential difference is $|e_j^1+g_j^l|+|e_j^2-g_j^l|-|e_j^1|-|e_j^2|$. The conformality of $g^l$ with $x^2-x^1=e^2-e^1$ shows that this is nonnegative. Furthermore, for $j=j^*$, this difference is strictly negative since $0<g_{j^*}^l\le g_\infty(A)<e_{j^*}^2-e_{j^*}^1$ and $e_{j^*}^1<0<e_{j^*}^2$. Therefore, the induction hypothesis applies and shows how to further modify the solutions to obtain a balanced set of solutions.
\end{proof}

Note that if $x=x^1+\dots+x^k$ is a decomposition of a solution $x\in\Z_{\ge0}^n$ to $Ax=kb$ as in the setting of the IDP, then \cref{lemma:balancing} shows that the resulting solutions $x^i\in\Z_{\ge0}^n$ to $Ax^i=b$ can be modified to be close to $\tfrac1kx$. In this case, the majorization guarantee ensures that all $x^i$ remain nonnegative and that $x=x^1+\dots+x^k$ is preserved. We are now ready to prove \cref{lemma:idp-dual-glueing}.

\begin{proofof}{lemma:idp-dual-glueing}
    We use $m=r+s_1+\dots+s_n$ and $n'=t_1+\dots+t_n$ to refer to total number of rows and columns of a constraint matrix $A\in\mathcal A$. In addition, we use $B=[B_1,\dots,B_n]$ to refer to the top $r$ rows of $A$. We first show \cref{claim:align-stacked}.

    \begin{claimin}
        \label{claim:align-stacked}
        There exists an integer $\overline k=\O(rU)^r$ so that the following holds: if $A\in\mathcal A$, $x\in\Z_{\ge0}^{n'}$, $b\in\Z^m$, and $k\in\Z_{\ge0}$ are such that $Ax=kNb$ and $k>\overline k$, then there exists an $\hat x\in\Z_{\ge0}^{n'}$ and a positive integer $k'<k$ satisfying $\hat x\le x$ and $A\hat x=k'Nb$.
    \end{claimin}

    \begin{claimproof}
        For notational convenience, we split the solution vector $x$ into $n$ bricks $x_1\in\Z_{\ge0}^{t_1},\dots$, $x_n\in\Z_{\ge0}^{t_n}$ so that $x=[x_1;\dots;x_n]$. Additionally, we split the right-hand side vector into $1+n$ bricks $b_0\in\Z^r,b_1\in\Z^{s_1},\dots,b_n\in\Z^{s_n}$ so that $b=[b_0;b_1;\dots;b_n]$.

        We choose $\overline k=(2(4r-2)U+1)^r=\O(rU)^r$. To show that this choice suffices, for each brick $i\in[n]$, we independently use the integer decomposition property of the system $D_ix_i=kNb_i$ to decompose $x_i$ into $x_i=y_i^1+\dots+y_i^k$ with $D_iy_i^j=Nb_i$, $y_i^j\in\Z_{\ge0}^{t_i}$ for $j\in[k]$.

        For each $i\in[n]$, the vectors $[y_i^j;-B_iy_i^j]$ for $j\in[k]$ are solutions to the system
        \[
            \begin{bmatrix}
                B_i&I\\
                D_i&
            \end{bmatrix}[y_i^j;-B_iy_i^j]=
            \begin{bmatrix}
                0\\
                Nb_i
            \end{bmatrix}
        \]
        Therefore, we can apply \cref{lemma:balancing} with respect to this composite constraint matrix which has an $\ell_\infty$-Graver complexity bounded by $U$. This allows us to assume that the error vectors $e_i^j=B_iy_i^j-\tfrac1kB_ix_i$ are bounded by $\|e_i^j\|_\infty\le U$ for $j\in[k]$.
        
        We apply \cref{lemma:colorful-steinitz-lemma} to permute the sequences $(e_i^j)_{j\in[k]}$ for $i\in[n]$. Note that the sum of each of the sequences is zero. In particular, $\sum_{j\in[k]}\sum_{i\in[n]}e_i^j=\veczero$. We find permutations $\pi_i\colon[k]\to[k]$ so that for any $\ell\in[k]$ we have that
        \[
            \Bigl\|\sum_{j\in[\ell]}\sum_{i\in[n]}e_i^{\pi_i(j)}\Bigr\|_\infty\le(4r-2)\cdot U.
        \]
        Note that the former sum is equal to
        \begin{align*}
            &\sum_{j\in[\ell]}\sum_{i\in[n]}e_i^{\pi_i(j)}=\sum_{j\in[\ell]}\sum_{i\in[n]}(B_iy_i^{\pi_i(j)}-\tfrac1kB_ix_i)=\sum_{j\in[\ell]}\Bigl(\bigl(\sum_{i\in[n]}B_iy_i^{\pi_i(j)}\bigr)-\tfrac1kBx\Big)\\
            &=\bigl(\sum_{j\in[\ell]}\sum_{i\in[n]}B_iy_i^{\pi_i(j)}\bigr)-\tfrac\ell kBx=\bigl(\sum_{j\in[\ell]}\sum_{i\in[n]}B_iy_i^{\pi_i(j)}\bigr)-\ell Nb_0.
        \end{align*}
        In particular, these sums are integral, and therefore lie within the discrete $\ell_\infty$-norm ball of radius $(4r-2)U$, which has $\overline k$ elements. Since $k>\overline k$, the pigeonhole principle implies that there must be two indices $j_1<j_2$ such that
        \[
            \sum_{j\in[j_1]}\sum_{i\in[n]}B_iy_i^{\pi_i(j)}-j_1\cdot Nb_0=\sum_{j\in[j_2]}\sum_{i\in[n]}B_iy_i^{\pi_i(j)}-j_2\cdot Nb_0.
        \]
        Let $J=\{j_1+1,j_1+2,\dots,j_2\}$, which satisfies $0<|J|<k$. We will now verify that $\hat x$ defined by $\hat x_i=\sum_{j\in J}y_i^{\pi_i(j)}$ for $i\in[n]$ is a suitable decomposition step satisfying $A\hat x=|J|\cdot Nb$. Note that integrality, nonnegativity of $\hat x$, and $\hat x\le x$ are immediate. From the collision of the prefix sums, we derive that
        \[
            B\hat x=\sum_{i\in[n]}B_i\hat x_i=\sum_{i\in[n]}\sum_{j\in J}B_iy_i^{\pi_i(j)}=\sum_{j\in J}\sum_{i\in[n]}B_iy_i^{\pi_i(j)}=|J|\cdot Nb_0.
        \]
        For the local constraints, we verify that $D_i\hat x_i=\sum_{j\in J}D_iy_i^{\pi_i(j)}=\sum_{j\in J}Nb_i=|J|\cdot Nb_i$ for all $i\in[n]$, as required.
    \end{claimproof}
    
    Let $\overline k=\O(rU)^r$ be the number appearing in \cref{claim:align-stacked}. Choose $M$ to be $N$ times the dilation factor $M'$ obtained from \cref{lemma:general-scaled-idp} for $d=1$ and $\Delta=\overline k$. We estimate $M=M'\cdot N=2^{\O(\overline k)}\cdot N=2^{\O(rU)^r}\cdot N$.

    Let $A\in\mathcal A$, $x\in\Z_{\ge0}^{n'}$, $b\in\Z^m$, and $k\in\Z_{\ge0}$ be such that $Ax=kMb=(kM')Nb$. Applying \cref{claim:align-stacked} to decompose $x$ exhaustively, we find vectors $\hat x^1,\dots,\hat x^\ell\in\Z_{\ge0}^{n'}$ so that $x=\hat x^1+\dots+\hat x^\ell$ and $A\hat x^j=k^jNb$ for some $k^j\in[\overline k]$ for $j\in[\ell]$ and $k^1+\dots+k^\ell=kM'$. The IDP of the polyhedron $\{y\in\R_{\ge0}^\ell:k^1y^1+\dots+k^\ell y^\ell=M'\}$ applied to the all-ones vector, which is in the $k$-dilation, now shows that these pieces $\{1,\dots,\ell\}$ can be partitioned into $k$ sets $J_l$ with $\sum_{j\in J_l}k^j=M'$ for each $l\in[k]$. For this reason, we set $x^l=\sum_{j\in J_l}\hat x^j$ for $l\in[k]$. We verify that
    \[
        x^1+\dots+x^k=\sum_{l\in[k]}\sum_{j\in J_l}\hat x^j=\sum_{j\in[\ell]}\hat x^j=x
    \]
    and that
    \[
        Ax^l=\sum_{j\in J_l}A\hat x^j=\sum_{j\in J_l}k^jNb=M'\cdot Nb=Mb
    \]
    for $l\in[k]$, showing that $A$ indeed has the IDP after an $M$-dilation.
\end{proofof}

We note that the qualitative result of \cref{lemma:idp-dual-glueing} can be obtained using the non-colorful version of the Steinitz lemma by viewing the block-diagonal part of an $n$-fold matrix as a single block, and directly using the Graver complexity bound for $n$-fold constraint matrices from~\cite{DBLP:conf/icalp/EisenbrandHK18} as $U$. However, in this case, $U$ will be slightly larger, resulting in a cubic top-level exponent in $M$ in \cref{corollary:n-fold-scaled-idp}.

\begin{corollary}
    \label{corollary:n-fold-scaled-idp}
    Let $r,s,\Delta\in\Z_{\ge0}$ be given. Then there exists a positive integer $M=2^{\O((r+s)\Delta)^{r(r+s)}}$ so that the class of $n$-fold matrices of the form (\ref{eq:n-fold-matrix}), where $B_i\in\{-\Delta,-\Delta+1,\dots,\Delta\}^{r\times t_i}$ and $D_i\in\{-\Delta,-\Delta+1,\dots,\Delta\}^{s\times t_i}$ for $i\in[n]$, has the IDP after an $M$-dilation.
\end{corollary}

\begin{proof}
    The matrices $D_i\in\{-\Delta,-\Delta+1,\dots,\Delta\}^{s\times t_i}$ have the IDP after an $N$-dilation for $N=2^{\O((\sqrt s\Delta)^s)}$ as a result of \cref{lemma:general-scaled-idp}. If $r=0$, we are done. Otherwise, apply \cref{lemma:idp-dual-glueing} and use that the $\ell_1$-Graver complexity, thus in particular, the $\ell_\infty$-Graver complexity, of the matrices
    \[
        \begin{bmatrix}
            B_i&I\\
            D_i
        \end{bmatrix}\in\{-\Delta,-\Delta+1,\dots,\Delta\}^{(r+s)\times(t_i+r)}
    \]
    is bounded by $U=\O((r+s)\Delta)^{r+s}$ as a result of the Graver complexity bound from~\cite{DBLP:conf/icalp/EisenbrandHK18}. This yields
    \[
        M=2^{\O(r\cdot \O((r+s)\Delta)^{r+s})^r}\cdot2^{\O((\sqrt s\Delta)^s)}=2^{\O((r+s)\Delta)^{r(r+s)}}.\qedhere
    \]
\end{proof}

\begin{corollary}
    \label{corollary:dual-treedepth-scaled-idp}
    Let $d,\Delta\in\Z_{\ge0}$. Then there exists a positive integer $M=2^{\O(\Delta)^{2^d}}$ so that the class of matrices consisting of $A\in\{-\Delta,-\Delta+1,\dots,\Delta\}^{m\times n}$ with dual treedepth bounded by $d$ has the IDP after an $M$-dilation.
\end{corollary}

\begin{proof}
    This follows from an inductive argument on $d$: for each $d\in\Z_{\ge0}$, we show that there exists a positive integer $f(d)$ so that that the class of matrices consisting of $A\in\{-\Delta,-\Delta+1,\dots,\Delta\}^{m\times n}$ with dual treedepth at most $d$ has the IDP after an $f(d)$-dilation. For $d=0$, the choice $f(d)=1$ suffices.
    
    Assume that the induction hypothesis is shown for some $d-1\ge0$, and let $A\in\{-\Delta,-\Delta+1,\dots,\Delta\}^{m\times n'}$ have dual treedepth at most $d\ge1$. Then the rows and columns of $A$ can be permuted so that it is of the form (\ref{eq:n-fold-matrix}) for $B_i\in\{-\Delta,-\Delta+1,\dots,\Delta\}^{r\times t_i}$ and $D_i\in\{-\Delta,-\Delta+1,\dots,\Delta\}^{s_i\times t_i}$ where all $D_i$ have dual treedepth of at most $d-r$. From the recursive definition of dual treedepth, we may assume that $r\in\{0,1\}$. Additionally, if $r=0$, then $A$ is block-diagonal and it suffices to show that all $D_i$ have the IDP after an $f(d)$-dilation. By considering a decomposition of the form (\ref{eq:n-fold-matrix}) of $A$ with a maximal number of blocks $n$ (of nonzero dimensions), each of these $D_i$ is either a zero matrix, which has the IDP after any dilation, or can be written in the form (\ref{eq:n-fold-matrix}) with $r=1$. Therefore, it is sufficient to consider the case where $r=1$.
    
    In this case, the the matrices $D_i$ have dual treedepth at most $d-1$, meaning that the induction hypothesis applies and guarantees that the matrices $D_i$ have the IDP after an $f(d-1)$-dilation. Additionally, the matrices
    \[
        \begin{bmatrix}B_i&I\\D_i\end{bmatrix}
    \]
    have dual treedepth at most $d$. Hence, the $\ell_1$-Graver complexity bound by Knop, Pilipczuk, and Wrochna~\cite{DBLP:journals/toct/KnopPW20} shows that we can apply \cref{lemma:idp-dual-glueing} for $U=\O(\Delta)^{2^d}$. Therefore, $A$ has the IDP after an $M$-dilation where
    \[
        M=2^{\O(U)}\cdot f(d)=2^{\O(\Delta)^{2^d}}\cdot f(d).
    \]
    Set $f(d+1)=M$. Solving the recursion for $f$ yields \cref{corollary:dual-treedepth-scaled-idp}.
\end{proof}

We note that any multiplier $M$ satisfying the conditions in \cref{corollary:dual-treedepth-scaled-idp} must be at least $2^{\Delta^{\Omega(2^d)}}$. This is discussed in \cref{appendix:lower-bounds}.

In addition to \cref{corollary:dual-treedepth-scaled-idp}, we can also derive bounds on dilations needed to establish the IDP for matrices with bounded primal treedepth, which includes $2$-stage stochastic matrices. Since the proof uses similar balancing ideas to that of the proof of \cref{lemma:idp-dual-glueing} and requires a lengthy modification of a proof of a lemma from~\cite{DBLP:conf/esa/CslovjecsekEPVW21}, we postpone the proof of \cref{lemma:idp-primal-glueing} and \cref{corollary:primal-treedepth-scaled-idp} to \cref{appendix:primal-treedepth-scaled-idp}.

\begin{lemma}
    \label{lemma:idp-primal-glueing}
    Let $N$, $U$, and $r$ be positive integers. Then there exists a positive integer $M=2^{(rU)^{\O(r^2)}}\cdot N$ so that the following holds: let $\mathcal A'$ be a class of matrices that has the IDP after an $N$-dilation. Let $\mathcal A$ be the class of matrices of the form (\ref{eq:2-stage-matrix}) for $C_i\in\Z^{t_i\times r}$ and $D_i\in\Z^{t_i\times s_i}$ for which $[C_i,D_i]\in\mathcal A'$ and $g_\infty([C_i,D_i])\le U$ for all $i\in[n]$. Then $\mathcal A$ has the IDP after an $M$-dilation.
\end{lemma}

\cref{corollary:primal-treedepth-scaled-idp} follows almost immediately from \cref{lemma:idp-primal-glueing} and the $\ell_\infty$-Graver complexity bound by Klein and Reuter~\cite{DBLP:journals/talg/KleinR24} for multistage stochastic matrices.

\begin{corollary}
    \label{corollary:primal-treedepth-scaled-idp}
    Let $d,\Delta\in\Z_{\ge0}$. Then there exists a positive integer $M=2^{2^{\Delta^{2^{\O(d)}}}}$ so that the class of matrices consisting of $A\in\{-\Delta,-\Delta+1,\dots,\Delta\}^{m\times n}$ with primal treedepth bounded by $d$ has the IDP after an $M$-dilation.
\end{corollary}

%% file: sections/periodic-convexity.tex
\subsection{Combining midpoint convexity, bounded Graver complexity and low-dimensional interaction}
\label{sec:periodic-convexity}

In this section, we first exploit Graver complexity bounds to modify pairs $Ax^1=r+Mz^1$, $Ax^2=r+Mz^2$ if the distance between $z^1$ and $z^2$ is large. By producing modified $\hat x^1,\hat x^2$ with associated $\hat z^1,\hat z^2$ that are closer together, we can effectively assume that the points $z^i$ in (\ref{eq:0-centered-convexity}) lie inside of a bounded box. We will later exploit this to effectively bound the number of points $k$ in (\ref{eq:0-centered-convexity}), after which we can apply the operation of \cref{lemma:idp-implies-midpoint-convexity} a bounded number of times to obtain the periodic convexity of value functions. Since the arguments of this section depend exponentially on the dimension of the variations of the right-hand sides, we will consider restricted value functions $h\colon\Z^d\to\R\cup\{\infty\}$ of the form $h(z)=\min\{f(x)\ \vert\ Ax=r+M[z;\veczero],\,x\in\Z^n\}$ for any $r\in\Z^m$. Despite these restrictions, this setting is still sufficiently expressive to develop the algorithms in \cref{sec:algorithms}.

We obtain the pairwise modification by inspecting the difference between two solutions $x^1,x^2$ and their corresponding right-hand sides $r+M[z^1;\veczero],r+M[z^2;\veczero]$. If $Mz^2-Mz^1$ is large, a single integer vector of bounded norm appears as a projection of a Graver basis element at least $M$ times. This yields a way to exchange a part of the solutions and decrease their right-hand side distance while simultaneously remaining on the grid $r+M(\Z^d\times\{\veczero\})$. This is formalized in \cref{claim:closing}. In \cref{lemma:reduction-to-bounded-box}, $\matzero$ stands for the zero matrix of appropriate dimension.

\begin{lemma}
    \label{lemma:reduction-to-bounded-box}
    Let $M,d,G\in\Z_{\ge0}$ and let $\mathcal A$ be a class of constraint matrices such that for any $A\in\mathcal A$, $A\in\Z^{m\times n}$ the augmented matrix
    \[
        A'=\begin{bmatrix}
            \multirow{2}{*}{A}&I\\
            &\matzero
        \end{bmatrix}\in\Z^{m\times(n+d)}
    \]
    satisfies $g_\infty(A')\le G$. Then there exists a positive integer $B=\O(G)^{d+1}$ so that the following holds: let $A\in\mathcal A$, $A\in\Z^{m\times n}$, $r\in\Z^m$, $x^1,\dots,x^k\in\Z^n$, and $z^1,\dots,z^k\in\Z^d$ be such that $z^1+\dots+z^k=\veczero$ and $Ax^i=r+M[z^i;\veczero]$ for $i\in[k]$, then there exist $\hat x^1,\dots,\hat x^k\in\Z^n$ and $\hat z^1,\dots,\hat z^k\in\Z^d$ so that $A\hat x^i=r+M[\hat z^i;\veczero]$ with $\|\hat z^i\|_\infty\le B$ for $i\in[k]$ and $[\hat x^1,\dots,\hat x^k]\rowmajorizedby[x^1,\dots,x^k]$.
\end{lemma}

\begin{proof}
    Let $F=\{g\in\Z^d:\|g\|_\infty\le G\}$ and choose $B=|F|\cdot G=\O(G)^{d+1}$. We first show how to close the gap between two $z^i$-s in the combination assuming they are further apart than $B$ in the $\ell_\infty$-norm. \cref{lemma:reduction-to-bounded-box} follows from \cref{claim:closing} and an inductive argument on the potential $\sum_{i\in[k]}\|z^i\|_1$ that is analogous to the one used in the proof of \cref{lemma:balancing}.
    
    \begin{claimin}
        \label{claim:closing}
        Let $r'\in\Z^m$, $x^A,x^B\in\Z^n, z\in\Z^d$, and $j^*\in[d]$ be such that $Ax^A=r'$, $Ax^B=r'+M[z;\veczero]$ and $|z_{j^*}|>B$. Then there exist $\hat x^A,\hat x^B\in\Z^n$ and a step $p\in\Z^d$ so that $p\sqsubseteq z$, $0<|p_{j^*}|<|z_{j^*}|$, $A\hat x^A=r'+M[p;\veczero]$, $A\hat x^B=r'+M[z-p;\veczero]$, and $[\hat x^ A,\hat x^B]\rowmajorizedby[x^A,x^B]$.
    \end{claimin}

    \begin{claimproof}
        Let $y=x^B-x^A$ and note that $[y;-Mz]$ is in the integer kernel of the composite matrix $A'$. Therefore, it can be decomposed into the sum of Graver basis elements $[y;-Mz]=[y^1;w^1]+\dots+[y^k;w^k]$ with $y^i\sqsubseteq y$, $y\in\Z^n$, $w^i\sqsubseteq-Mz$, $Ay^i=[-w^i;\veczero]$, and $w^i\in F$ for $i\in[k]$. We group the decomposed elements by their projection on the last $d$ coordinates $w^i$ and define $I_g=\{i\in[k]:w^i=g\}$. If there does not exist a $g\in F$ with $|I_g|\ge M$ and $g_{j^*}\ne0$, then we have that
        \[
            |z_{j^*}|=\Bigl|\tfrac1M\sum_{g\in F}|I_g|\cdot g_{j^*}\Bigr|=\tfrac1M\sum_{g\in F:g_{j^*}\ne0}|I_g|\cdot|g_{j^*}|\le\tfrac1M\cdot|F|\cdot M\cdot G=B,
        \]
        which contradicts $|z_{j^*}|>B$. Therefore, we can find a size $M$ subset $I_g'$ of a set $I_g$ for some $g\in F$ for which $g_{j^*}\ne0$. We define the step $s=\sum_{i\in I_g'}y^i$ with corresponding right-hand side modification $p=-g$. It holds that $Mg\sqsubseteq-Mz$, which implies $p=-g\sqsubseteq z$. Furthermore, $0<|p_j^*|=|g_j^*|\le G\le B<|z_{j^*}|$. Define $\hat x^A=x^A+s$ and $\hat x^B=x^B-s$. The sign-compatibility of the steps $y^i$ with $x^B-x^A$ shows that $[\hat x^A,\hat x^B]\rowmajorizedby[x^ A,x^B]$. Moreover, observe that
        \[
            A\hat x^A=Ax^A+As=r'+\sum_{i\in I_g'}Ay^i=r'+\sum_{i\in I_g'}[-w^i;\veczero]=r'+M[-g;\veczero]=r'+M[p;\veczero]
        \]
        and similarly that $A\hat x^B=r'+M[z-p;\veczero]$.
    \end{claimproof}

    We now provide the inductive argument on the potential $\sum_{i\in[k]}\|z^i\|_1$. Note that if the potential is at most $B$, $\hat z^i=z^i$ satisfies the conditions of the lemma. Hence, we may focus on the induction step. If there is no $j^*\in[d]$ and $i\in[k]$ with $|z_{j^*}^i|>B$, we are done. Otherwise, w.l.o.g., assume that $z_{j^*}^1<-B$ (the other sign is treated analogously). Since $z^1+\dots+z^k=\veczero$, there must exist another point $z^{i'}$, $i'\in[k]$ with $z_{j^*}^{i'}>0$. W.l.o.g., this point is $z^2$. We now apply \cref{claim:closing} with $r'=r+M[z^1;\veczero]$ and $z=z^2-z^1$, which satisfies $z_{j*}\ge B+2>B$, to find an integral step $p$ conformal to $z^2-z^1$ with $0<p_{j^*}<z_{j^*}$. Set $\hat z^1=z^1+p$ and $\hat z^2=z^2-p$. \cref{claim:closing} additionally provides updated solutions $\hat x^1,\hat x^2\in\Z^t$ that satisfy $[\hat x^1,\hat x^2]\rowmajorizedby[x^1,x^2]$ and $A\hat x^1=r'+M[p;\veczero]=r+M[z^1;\veczero]+M[p;\veczero]=r+M[\hat z^1;\veczero]$, $A\hat x^2=r'+M[z-p;\veczero]=r+M[\hat z^2;\veczero]$. Now, consider the collections $\hat x^1,\hat x^2,x^3,\dots,x^k$ and $\hat z^1,\hat z^2,z^3,\dots,z^k$. We have that $[\hat x^1,\hat x^2,x^3,\dots,x^k]\rowmajorizedby[x^1,\dots,x^k]$ and $\hat z^1+\hat z^2+z^3+\dots+z^k=\veczero$. We now inspect the change $\|\hat z^1\|_1+\|\hat z^2\|_1-\|z^1\|_1-\|z^2\|_1$ in the potential. The contribution of the $j$-th coordinate to this difference is exactly $|z_j^1+p_j|+|z_j^2-p_j|-|z_j^1|-|z_j^2|$. The conformality of $p_j$ with $z_j^2-z_j^1$ shows that this is nonnegative. Furthermore, for $j=j^*$, this difference is strictly negative because $0<p_j^*<z_{j^*}^2-z_{j^*}^1$ and $z_{j^*}^1<0<z_{j^*}^2$. Therefore, the induction hypothesis applies and shows how to further modify the solutions and right-hand side vectors $z^i$ so that $\|z^i\|_\infty\le B$.
\end{proof}

We can now use \cref{lemma:reduction-to-bounded-box} to move all $z^i$ appearing in (\ref{eq:0-centered-convexity}) into a bounded box. In this situation, any multiset of points summing to zero must contain bounded submultisets satisfying the same property. See \cref{fig:submultisets}. We will treat these submultisets in isolation, which effectively allows us to assume that the number $k$ in (\ref{eq:0-centered-convexity}) is bounded. In this way, we can apply \cref{lemma:idp-implies-midpoint-convexity} a bounded number of $\ell$ times to obtain solutions to $Ax=r$. Here, we ensure that the we meet the conditions of \cref{lemma:idp-implies-midpoint-convexity} by applying a $2^\ell$-dilation to the grid defined by the dilation factor that yields the IDP. This is visualized in \cref{fig:averaging}.

\begin{figure}
    \begin{subfigure}[t]{0.45\textwidth}
        \centering
        \includegraphics[width=\textwidth]{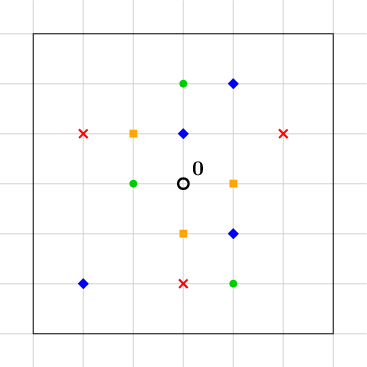}
        \caption{The multiset of bounded points with zero sum can be partitioned into submultisets of bounded cardinality with zero sum.}
        \label{fig:submultisets}
    \end{subfigure}
    \hspace{0.05\textwidth}
    \begin{subfigure}[t]{0.45\textwidth}
        \centering
        \includegraphics[width=\textwidth]{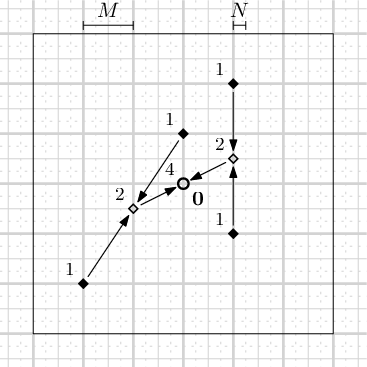}
        \caption{The grid is dilated so that we can apply midpoint convexity and take pairwise averages of the right-hand side vectors until $z=\veczero$. The number of solutions per right-hand side vector is shown on the top left of the points.}
        \label{fig:averaging}
    \end{subfigure}
    \caption{Visualization of the $z^i$-s in the two steps in the proof of \cref{claim:vector-rearrangement-with-insertion}.}
    \label{fig:vector-rearrangement-with-insertion}
\end{figure}

\begin{lemma}
    \label{lemma:fixed-dimension-periodic-convexity}
    Let $d,G\in\Z_{\ge0}$ and let $\mathcal A$ be a class of constraint matrices that is closed under inverting columns, that has the IDP after an $N$-dilation, and assume that for any $A\in\mathcal A$, $A\in\Z^{m\times n}$ the augmented matrix
    \[
        A'=\begin{bmatrix}
            \multirow{2}{*}{A}&I\\
            &\matzero
        \end{bmatrix}\in\Z^{m\times(n+d)}
    \]
    satisfies $g_\infty(A')\le G$. Then there exists an integer $M=\O(G)^{d(d+1)}\cdot N$ so that the value function $h\colon\Z^d\to\R\cup\{-\infty,\infty\}$ defined by $h(z)=\min\{f(x)\ \vert\ Ax=r+M[z;\veczero],\,x\in\Z^n\}$ is convex extensible for any $A\in\mathcal A$, $A\in\Z^{m\times n}$, separable convex $f\colon\R^n\to\R\cup\{\infty\}$, and $r\in\Z^m$.
\end{lemma}

\begin{proof}
    Let $A\in\mathcal A$, $A\in\Z^{m\times n}$, and $r\in\Z^m$. We first show \cref{claim:vector-rearrangement-with-insertion}, in which we introduce copies of a dummy point $x^*$, restricting to point multisets with cardinalities that are a power of $2$. After, we resolve the technicalities that these introduce and cover degenerate cases.
    
    \begin{claimin}
        \label{claim:vector-rearrangement-with-insertion}
        Let $x^1,\dots,x^k\in\Z^n$, and $z^1,\dots,z^k\in\Z^d$ be such that $z^1+\dots+z^k=\veczero$ and $Ax^i=r+M[z^i;\veczero]$ for all $i\in[k]$. Assume that there exists an $x^*\in\Z^n$ so that $Ax^*=r$ and denote $x^i=x^*$ for $i>k$. Then there exists an integer $\overline k\ge k$ and points $\hat x^1,\dots,\hat x^{\overline k}\in\Z^n$ so that $A\hat x^1=\dots=A\hat x^{\overline k}=r$ and $[\hat x^1,\dots,\hat x^{\overline k}]\rowmajorizedby[x^1,\dots,x^{\overline k}]$.
    \end{claimin}

    \begin{claimproof}
        Let $B=\O(G)^{d+1}$ be the number from \cref{lemma:reduction-to-bounded-box}. We pick $M=2^{\ell}\cdot N$ for some $\ell=d\log(dB)+\O(d)$, which will be elaborated upon later. We find $M=2^{d\log(d\cdot\O(G)^{d+1})+\O(d)}\cdot N=\O(G)^{d(d+1)}\cdot N$.
    
        First, we apply \cref{lemma:reduction-to-bounded-box} to find $\tilde x^1,\dots,\tilde x^k$, $\tilde z^1,\dots,\tilde z^k$ such that $A\tilde x^i=r+M[\tilde z^i;\veczero]$, $\|\tilde z^i\|_\infty\le B$ for $i\in[k]$ and $[\tilde x^1,\dots,\tilde x^k]\rowmajorizedby[x^1,\dots,x^k]$ together with $\tilde z^1+\dots+\tilde z^k=\veczero$. Thus, the all-ones vector is in the kernel of the matrix $\tilde Z=[\tilde z^1,\dots,\tilde z^k]\in\{-B,-B+1,\dots,B\}^{d\times k}$ and can be decomposed into conformal Graver basis elements of $\tilde Z$, which have an $\ell_1$-norm of at most $W:=\O(dB)^d$~\cite{DBLP:conf/icalp/EisenbrandHK18}. Since all these vectors are binary, this can be interpreted as a partition of the multiset $(\tilde x^1,\tilde z^1),\dots,(\tilde x^k,\tilde z^k)$ into submultisets, each of at most $W$ pairs $(\tilde x^i,\tilde z^i)$, so that the $\tilde z^i$-s within a submultiset sum to zero.
    
        Consider an arbitrary such submultiset $(\overline x^1,\overline z^1),\dots,(\overline x^{\tilde k},\overline z^{\tilde k})$ with $\overline z^1+\dots+\overline z^{\tilde k}=\veczero$ and $\tilde k\le W$. Let $\ell=\lceil\log W\rceil=d\log(dB)+\O(d)$. Conceptually, we add $2^\ell-\tilde k$ copies of $(x^*,\veczero)$ to the multiset and denote $(\overline x^i,\overline z^i)=(x^*,\veczero)$ for $i>\tilde k$. This shows that we can pick $\overline k\ge k$ so that, after setting $\tilde x^i=x^i=x^*$ and $\tilde z^i=\veczero$ for $i>\overline k$, the multiset $(\tilde x^1,\tilde z^1),\dots,(\tilde x^{\overline k},\tilde z^{\overline k})$ partitions into submultisets of exactly $2^\ell$ pairs each so that the $\tilde z^i$-s within a submultiset sum to zero. Therefore, in order to prove \cref{claim:vector-rearrangement-with-insertion}, it suffices to show that for any such submultiset $(x^1,z^1),\dots,(x^{2^\ell},z^{2^\ell})$, we can construct $\hat x^1,\dots,\hat x^{2^\ell}\in\Z^n$ satisfying $A\hat x^1=\dots=A\hat x^{2^\ell}=r$ and $[\hat x^1,\dots,\hat x^{2^\ell}]\rowmajorizedby[x^1,\dots,x^{2^\ell}]$.
    
        To accomplish this, we employ \cref{lemma:idp-implies-midpoint-convexity} to take pairwise averages as visualized in \cref{fig:averaging}. To make this precise, we will, for each level $l\in\{0,1,\dots,\ell\}$, construct a partition $\mathcal I_l$ of the indices $[2^\ell]$ of the submultiset into exactly $2^{\ell-l}$ blocks, each of cardinality exactly $2^l$. We associate an average right-hand side vector $z_I=\tfrac1{|I|}\sum_{i\in I}z^i$ with each block $I\subseteq[2^\ell]$. Additionally, for each block $I=\{i^1,\dots,i^{|I|}\}$, we will construct a multiset $X_I=\{x_I^1,\dots,x_I^{|I|}\}$ of $|I|$ integral solutions to $Ax=r+M[z_I;\veczero]$ satisfying $[x_I^1,\dots,x_I^{|I|}]\rowmajorizedby[x^{i^1},\dots,x^{i^{|I|}}]$.
        
        We show this by induction on the level $l\in\{0,1,\dots,\ell\}$. For level $l=0$, we start with the singleton partition $\mathcal I_0=\{\{1\},\{2\},\dots,\{2^\ell\}\}$ and $X_{\{i\}}=\{x^i\}$ for $i\in[2^\ell]$. Now assume that we have constructed a suitable partition for a level $l\in\{0,\dots,\ell-1\}$. To construct a partition $\mathcal I_{l+1}$, we arbitrarily merge pairs of blocks from the lower level $l$. This yields a partition of the correct dimensions. It is left to show how to obtain the multisets $X_I$ for this partition. Let $I=I_A\cup I_B\in\mathcal I_{l+1}$ be a newly constructed block from distinct blocks $I_A=\{i_A^1,\dots,i_A^{2^l}\},I_B=\{i_B^1,\dots,i_B^{2^l}\}\in\mathcal I_l$. We arbitrarily pair up each solution $x_{I_A}^i$ from $X_{I_A}=\{x_{I_A}^1,\dots,x_{I_A}^{2^l}\}$ with a solution $x_{I_B}^i$ from $X_{I_B}=\{x_{I_B}^1,\dots,x_{I_B}^{2^l}\}$ and construct two new solutions $\hat x_A^i,\hat x_B^i$ to $Ax=r+Mz_I$ satisfying $[\hat x_A^i,\hat x_B^i]\rowmajorizedby[x_{I_A}^i,x_{I_B}^i]$. We obtain such points $\hat x_A^i$ and $\hat x_B^i$ by applying \cref{lemma:idp-implies-midpoint-convexity} with dilation $N$, $r'=r+M[z_I;\veczero]$, $z=2^\ell(z_{I_B}-z_I)$ and points $x_{I_A}^i,x_{I_B}^i$. Note that $z_I=\tfrac12(z_{I_A}+z_{I_B})$, $2^\ell(z_{I_B}-z_I)\in2^\ell\cdot\tfrac1{2^{l+1}}\Z^d\subseteq\Z^d$, and $Mz_I\in\Z^d$ as a result of the dimensions of the partition. Furthermore, it indeed holds that
        \begin{align*}
            Ax_{I_A}^i&=r+M[z_{I_A};\veczero]=r+M[2z_I-z_{I_B};\veczero]=r+M[z_I;\veczero]+M[z_I-z_{I_B};\veczero]\\
            &=r'+2^\ell\cdot N[z_I-z_{I_B};\veczero]=r'+N[-z;\veczero]
        \end{align*}
        and that $Ax_{I_B}^i=r'+2^\ell\cdot N[z_{I_B}-z_I;\veczero]=r'+N[z;\veczero].$
        By using the pairwise majorization from \cref{lemma:idp-implies-midpoint-convexity}, permuting vectors, and using the assumed majorization for $X_{I_A}$ and $X_{I_B}$, we find that
        \begin{align*}
            &\bigl[\hat x_A^1,\hat x_B^1,\dots,\hat x_A^{2^l},\hat x_B^{2^l}\bigr]\rowmajorizedby\bigl[x_{I_A}^1,x_{I_B}^1,\dots,x_{I_A}^{2^l},x_{I_B}^{2^l}\bigr]\\
            &\rowmajorizedby\bigl[x_{I_A}^1,\dots,x_{I_A}^{2^l},x_{I_B}^1,\dots,x_{I_B}^{2^l}\bigr]\rowmajorizedby\bigl[x^{i_A^1},\dots,x^{i_A^{2^l}},x^{i_B^1},\dots,x^{i_B^{2^l}}\bigr],
        \end{align*}
        as required. Now, using induction, we can obtain a partition for level $l=\ell$ with the desired properties. This yields the multiset $X_{[2^\ell]}=\{\hat x^1,\dots,\hat x^{2^\ell}\}$ of solutions to $Ax=r$ that is majorized by the original solutions as required.
    \end{claimproof}

    Let $f\colon\R^n\to\R\cup\{\infty\}$ be separable convex. We will now use \cref{claim:vector-rearrangement-with-insertion} to show that $h$ is convex extensible. It is well-known that $h$ is convex extensible on $\Z^d$ if and only if
    \begin{equation}
        \label{eq:weighted-convexity}
        h\Bigl(\sum_{i\in[\underline k]}\smash{\underline\lambda}^iz^i\Bigr)\le\sum_{i\in[\underline k]}\smash{\underline\lambda}^ih(z^i)
    \end{equation}
    holds when $\underline\lambda\in\R_{>0}^{\underline k}$ are positive convex multipliers, i.e., $\sum_{i\in[\underline k]}\underline\lambda^i=1$, and $z^1,\dots,z^{\underline k}\in\Z^d$ are points with an integral average $z^*:=\sum_{i\in[\underline k]}\smash{\underline\lambda}^iz^i\in\Z^d$ and when $h(z^i)<\infty$ for $i\in[\underline k]$. Since we aim to prove the lemma for any arbitrary $r\in\Z^m$, we may do a change of variables and consider the function $h'(z)=h(z+z^*)=\min\{f(x)\ \vert\ Ax=(r+M[z^*;\veczero])+M[z;\veczero],\,x\in\Z^n\}$ instead so that $h'(\veczero)\le\sum_{i\in[\underline k]}\smash{\underline\lambda}^ih'(z^i-z^*)$ is equivalent to (\ref{eq:weighted-convexity}). Therefore, we may assume that $z^*=\veczero$ without loss of generality, i.e., it suffices to show that
    \begin{equation}
        \label{eq:0-centered-weighted-convexity}
        h(\veczero)\le\sum_{i\in[\underline k]}\smash{\underline\lambda}^ih(z^i)
    \end{equation}
    for any points $z^1,\dots,z^{\underline k}\in\Z^d$ so that $z^1+\dots+z^{\underline k}=\veczero$. We first treat the case where $h(z^i)>-\infty$ for all $i\in[\underline k]$, i.e., all corresponding optimization problems are feasible and bounded.

    Observe that $\underline\lambda$ is a solution to the linear program
    \begin{equation}
        \label{eq:convex-combination-lp}
        \begin{aligned}
            & \text{minimize} & \sum_{i\in[\underline k]}h(z^i)\lambda^i, &\\
            & \text{subject to} \quad & \sum_{i\in[\underline k]}z^i\lambda^i & =\veczero,\\
            & & \sum_{i\in[\underline k]}\lambda^i & =1,\\
            & & \lambda & \ge\veczero.
        \end{aligned}
    \end{equation}
    Let $g(\lambda)$ denote the objective value $g(\lambda)=\sum_{i\in[\underline k]}h(z^i)\lambda^i$ of (\ref{eq:convex-combination-lp}). We need to show that $h(\veczero)\le g(\smash{\underline\lambda})$. The polyhedron underlying (\ref{eq:convex-combination-lp}) is rational, which shows that there exists a rational solution $\tilde\lambda$ to (\ref{eq:convex-combination-lp}) satisfying $g(\tilde\lambda)\le g(\smash{\underline\lambda})$. Let $k\in\Z_{\ge1}$ be an integer so that $k\cdot\tilde\lambda$ is integral. In this way, $g(\tilde\lambda)$ is seen to be the unweighted average of $k$ terms of the form $h(z^i)$, where $z^i$ appears $k\tilde\lambda^i$ times. Let $\tilde z^1,\dots,\tilde z^k$ denote these integer vectors with their appropriate multiplicities. Thus, it suffices to show that $h(\veczero)\le\tfrac1k\sum_{i\in[k]}h(\tilde z^i)=g(\tilde\lambda)$ as this implies $h(\veczero)\le g(\underline\lambda)$. Since all $h(z^i)$ are assumed to be finite, there are solutions $x^1,\dots,x^k\in\Z^n$ to $Ax^i=r+M[\tilde z^i;\veczero]$ attaining $f(x^i)=h(\tilde z^i)$ for $i\in[k]$.

    Write $f(x)=\sum_{j\in[n]}f_j(x_j)$ for univariate convex $f_j\colon\R\to\R\cup\{\infty\}$. For the sake of showing $h(\veczero)\le\tfrac1k\sum_{i\in[k]}h(\tilde z^i)$, we may assume that each $f_j$ is a nonnegative piecewise linear function with integral breakpoints and effective domain
    \[
        [L_j,R_j]=\bigl[\min_{i\in[k]}x_j^i,\max_{i\in[k]}x_j^i\bigr],
    \]
    i.e., $f_j(x_j)=\infty$ if and only if $x_j\notin[L_j,R_j]$. We now construct an alternative separable convex function $\tilde f\colon\R^n\to\R$ that is finite everywhere. We can define the functions $\tilde f_j\colon\R\to\R$ by setting
    \[
        \tilde f_j(x_j)=\begin{cases}
            f_j(x_j),&\text{if $L_j\le x_j\le R_j$,}\\
            f_j(L_j)+S(L_j-x_j),&\text{if $x_j<L_j$,}\\
            f_j(R_j)+S(x_j-R_j),&\text{if $x_j>R_j$,}
        \end{cases}
    \]
    for some large slope $S$ and let $\tilde f(x)=\sum_{j\in[n]}\tilde f_j(x_j)$. We pick $S$ sufficiently large so that each $\tilde f_j$ is convex and so that for any points $x,y\in\Z^n$ for which $f(y)<\infty$, but $f(x)=\infty$, we have that $\tilde f(y)<\tilde f(x)$.

    We now first argue that there exists an integral solution $x^*\in\Z^n$ to $Ax^*=r$. To see this, note that $Ax^1=r+M[\tilde z^1;\veczero]$ and $A\cdot\tfrac1k\sum_{i\in[k]}x^i=r$. Thus, $M[\tilde z^1;\veczero]=A(x^1-\tfrac1k\sum_{i\in[k]}x^i)$ is in the row space of $A$. We can invert the sign of some columns of $A$ to obtain a matrix $A'$ and nonnegative $y'\in\R_{\ge0}^n$ such that $A'y'=M[\tilde z^1;\veczero]$. By assumption, $A'$ has the IDP after an $N$ dilation and $M$ is a multiple of $N$. Thus, the polyhedron $\{y'\in\R_{\ge0}^n:A'y'=N\cdot (\tfrac MN\cdot[\tilde z^1;\veczero])\}$ has the IDP. Since integrality of a polyhedron is a necessary condition for it to have the IDP, there exists an integral $\tilde y'\in\Z_{\ge0}^n$ such that $A'\tilde y'=M[\tilde z^1;\veczero]$. Reverting the sign inversion shows that there is an integral $\tilde y\in\Z^n$ so that $A\tilde y=M[\tilde z^1;\veczero]$, meaning that $x^1-\tilde y\in\Z^n$ satisfies $A(x^1-\tilde y)=r$.
    
    Continuing, we may assume that $h(\veczero)>-\infty$, since (\ref{eq:0-centered-convexity}) holds trivially otherwise. Now we let $x^*\in\Z^n$ be a solution to $Ax^*=r$ that minimizes $\tilde f(x^*)$. We apply \cref{claim:vector-rearrangement-with-insertion} to the points $x^1,\dots,x^k$ and $x^*$. This yields an integer $\overline k\ge k$ and points $\hat x^1,\dots,\hat x^{\smash{\overline k}}\in\Z^n$ so that $A\hat x^i=r$ for all $i\in[\overline k]$ and $\tilde f(\hat x^1)+\dots+\tilde f(\hat x^{\smash{\overline k}})\le\tilde f(x^1)+\dots+\tilde f(x^{\smash{\overline k}})$. W.l.o.g., assume that $\hat x^1$ is the point minimizing $\tilde f(\hat x^i)$, which satisfies $\tilde f(\hat x^1)\le\tfrac1{\,\overline k\,}\sum_{i\in[\overline k]}\tilde f(x^i)$. This implies that
    \begin{align*}
        &\overline k\tilde f(x^*)\le\overline k\tilde f(\hat x^1)\le\sum_{i\in[\overline k]}\tilde f(x^i)=\sum_{i\in[k]}\tilde f(x^i)+(\overline k-k)\tilde f(x^*)\\
        \implies&k\tilde f(x^*)\le\sum_{i\in[k]}\tilde f(x^i)\\
        \implies&\tilde f(x^*)\le\tfrac1k\sum_{i\in[k]}\tilde f(x^i).
    \end{align*}
    In particular, $\tilde f(x^*)\le\max_{i\in[k]}\tilde f(x^i)=\tilde f(y)$ for some $y=x^i$, $i\in[k]$. Since $f(x^i)<\infty$ for all $i\in[k]$, we find that $f(x^*)<\infty$ due to our choice of $S$. But then also
    \[
        h(\veczero)\le f(x^*)=\tilde f(x^*)\le\tfrac1k\sum_{i\in[k]}\tilde f(x^i)=\tfrac1k\sum_{i\in[k]}f(x^i)=g(\tilde\lambda)
    \]
    follows because $\tilde f$ and $f$ agree when $f$ is finite, completing the argument for the case where $h(z^i)>-\infty$ in (\ref{eq:0-centered-weighted-convexity}).

    If there is a $z^{i^*}$ in (\ref{eq:0-centered-weighted-convexity}) so that $h(z^{i^*})=-\infty$, we need to argue that $h(\veczero)=-\infty$. This can shown via the finite case, again by using a different separable convex objective function $\tilde f_l(x)=\sum_{j\in[n]}\max\{l,f_j(x_j)\}$. Note that $\tilde h_l(z^{i^*})=\min\{\tilde f_l(x)\ \vert\ Ax=r+M[z^{i^*};\veczero],\,x\in\Z^n\}$ converges to $-\infty$ as $l\to-\infty$. Now, convex extensibility applied to $\tilde h_l$ shows that $h(\veczero)\le\tilde h_l(\veczero)\le\sum_{i\in[\underline k]}\smash{\underline\lambda}^i\tilde h_l(z^i)\le\smash{\underline\lambda}^{i^*}\tilde h_l(z^{i^*})+\sum_{i\in[\smash{\underline k}]\setminus\{i^*\}}\smash{\underline\lambda}^i\tilde h_l(z^i)$, where the former term converges to $-\infty$ as $l\to-\infty$ and the latter term is bounded from above.
\end{proof}
We note that periodic convexity can also be shown without introducing the dummy points $x^*$ and instead taking pairwise weighted averages using \cref{proposition:idp-iff-convex-extensibility-along-line}. However, this appears to yield a weaker bound on $M$ because it seems more difficult to bound the fractionality of the intermediary points $z_I$ when taking the weighted averages, which results in needing a dilation that is significantly larger than $2^\ell$ in \cref{claim:vector-rearrangement-with-insertion}.

We can now piece together the lemmata of this section to derive \cref{theorem:general-convex-extensibility}.

\theoremgeneralconvexextensibility*

\begin{proof}
    We obtain $N=2^{\O((\sqrt m\Delta)^m)}$ from \cref{lemma:general-scaled-idp}. We apply \cref{lemma:fixed-dimension-periodic-convexity} with $d=m$ and $G=\O(m\Delta)^m$~\cite{DBLP:conf/icalp/EisenbrandHK18} to find that $M=(\O(m\Delta)^m)^{m(m+1)}\cdot2^{\O((\sqrt m\Delta)^m)}=2^{\O((\sqrt m\Delta)^m)}$ suffices.
\end{proof}
Note that any $M$ satisfying the conditions of \cref{theorem:general-convex-extensibility} must be at least $2^{\Omega(\Delta^m)}$ as discussed in \cref{appendix:lower-bounds}.

%% file: sections/algorithms.tex
\section{Algorithms}
\label{sec:algorithms}

We will now exploit the periodic convexity of value functions of (block-structured) integer programs to obtain FPT algorithms. Each of our algorithms is obtained using the algorithmic framework of \cref{lemma:fixed-phase-value-function-reformulation}, which relies on fixed-phase value function reformulations. The framework covers integer programming problems that consist of a fixed number $p$ of arbitrary first-stage variables that, through $d$ constraints, interact with $n$ second-stage variables. Since $n$ is allowed to be large, we require that the constraint matrix of the second-stage variables induces tractable integer programs with a value function that is periodically convex with a bounded period $M$.

In the proof of \cref{lemma:fixed-phase-value-function-reformulation}, we move the objective value of the second-stage problem into the objective on the first-stage variables and follow the remainder guessing strategy by Cslovjecsek et al.~\cite{DBLP:journals/theoretics/CslovjecsekKLPP25}: we guess and fix the remainder of each first-stage variable modulo $M$. Such a guess also fixes the remainder modulo $M$ of the arguments of the value functions in the value function reformulation. The periodic convexity of the value function of the second-stage problem will then ensure that we obtain a low-dimensional integer program with a convex extensible objective function, which can be optimized using the algorithm by Veselov et al.~\cite{DBLP:journals/dam/VeselovGZC20}. Their algorithm can efficiently minimize discrete convic objective functions over integral variables, even when the objective function is only accessible through a comparison oracle. In particular, their algorithm applies to finite convex extensible functions, which are discrete convic.

\begin{lemma}
    \label{lemma:fixed-phase-value-function-reformulation}
    Let $A\in\Z^{m\times n}$, $p,d,U\in\Z_{\ge0}$, and $M\in\Z_{\ge1}$ be such that:
    \begin{itemize}
        \item The value function $h\colon\Z^d\to\R\cup\{-\infty,\infty\}$ given by
        \begin{equation}
            \label{eq:value-function-assumed-convex-extensibility}
            h(z)=\min\bigl\{f(x)\bigm\vert[A,I]x=r+M[z;\veczero],\,x\in\Z^{n+m}\bigr\}
        \end{equation}
        is convex extensible on $\Z^d\cap S$ for any $r\in\Z^m$, linear subspace $S$ with dimension at most $p$, and separable convex objective function $f\colon\R^{n+m}\to\R\cup\{\infty\}$.
        \item An optimal solution to
        \begin{equation}
            \label{eq:generic-second-stage-ip}
            \min\bigl\{f(x)\bigm\vert[A,I]x=b,\,l\le x\le u,\,x\in\Z^{n+m}\bigr\},
        \end{equation}
        can be found in time $T_{\mathrm{opt}}(\sigma)$ for any $l,u\in\Z^{n+m}$, $b\in\Z^m$, and separable convex $f\colon\R^{n+m}\to\R$ assuming that: an initial feasible solution is provided, $\|u-l\|_\infty\le\sigma$, and $f$ is accessible through a comparison oracle on $\Z^{n+m}$.
        \item The product $Ax\in\Z^m$ can be computed in time $T_{\mathrm{mul}}$ for any $x\in\Z^n$.
        \item The induced matrix norm $\|A\|_\infty$ is bounded by $U$ from above.
    \end{itemize}
    Let $C\in\Z^{d\times p}$, $l,u\in\Z^n$, $b\in\Z^m$, $V\in\R^{\ell\times p}$, $w\in\R^\ell$, $c\in\R^p$, and $\rho>0$ be such that $\{y\in\R^p:Vy\le w\}\subseteq\{y\in\R^p:\|y-c\|_\infty\le\rho\}$. Let $g\colon\R^p\to\R$ be convex and $f\colon\R^n\to\R$ be separable convex and let both be accessible through a comparison oracle on $\Z^{p+n}$. Then an optimal solution to
    \begin{equation}
        \label{eq:generic-2-stage-ip}
        \min\bigl\{g(y)+f(x)\bigm\vert[Cy;\veczero]+Ax=b,\,Vy\le w,\,y\in\Z^p,\,l\le x\le u,\,x\in\Z^n\bigr\}
    \end{equation}
    can be found in time
    \[
        M^p\cdot2^{\O(p^2\log p)}\cdot\Bigl(\log^{\O(1)}(\rho)+\log(\rho)\cdot\bigl(T_{\mathrm{opt}}((1+2U)\|u-l\|_\infty)+T_{\mathrm{mul}}+m+\ell\bigr)\Bigr).
    \]
\end{lemma}

\begin{proof}
    We first ensure that the second-stage subproblem is always feasible by verifying that $l\le u$ and introducing $m$ dummy variables $s$. Instead of solving (\ref{eq:generic-2-stage-ip}) directly, we will find an optimal solution to
    \begin{equation}
        \label{eq:feasible-generic-2-stage-ip}
        \begin{aligned}
            \min\bigl\{g(y)+f(x)+\delta(s)\bigm\vert{}&[Cy;\veczero]+Ax+s=b,\,Vy\le w,\\
            &y\in\Z^p,\,l\le x\le u,\,x\in\Z^n,\,s\in\Z^m\bigr\},
        \end{aligned}
    \end{equation}
    where $\delta(s)=\alpha\cdot\|s\|_1$ is a separable convex penalty function for some sufficiently large $\alpha>0$. More specifically, we guarantee that $g(y^1)+f(x^1)+\delta(s^1)<g(y^2)+f(x^2)+\delta(s^2)$ if $\|s^1\|_1<\|s^2\|_1$ for any feasible $(y^1,x^1)$ and $(y^2,x^2)$ in (\ref{eq:generic-2-stage-ip}). Note that the value of $\alpha$ does not need be computed because the comparison oracle for the objective function of (\ref{eq:feasible-generic-2-stage-ip}) can be implemented by lexicographically comparing $\|s\|_1$ and the original objective. The original problem (\ref{eq:generic-2-stage-ip}) is feasible if and only if every optimal solution to (\ref{eq:feasible-generic-2-stage-ip}) satisfies $s=0$, in which case an optimal solution to (\ref{eq:generic-2-stage-ip}) can be obtained by projecting an arbitrary optimal solution to (\ref{eq:feasible-generic-2-stage-ip}) onto the $y$ and $x$ variables.

    To solve (\ref{eq:feasible-generic-2-stage-ip}), we guess the remainder of each of the coordinates of $y$ modulo $M$, i.e., we guess the vector $\hat r\in\{0,1,\dots,M-1\}^p$ so that there is an optimal solution to (\ref{eq:feasible-generic-2-stage-ip}) of the form $y=\hat r+Mv$ for some $v\in\Z^p$. Substitution in (\ref{eq:feasible-generic-2-stage-ip}) yields the equivalent fixed phase problem
    \begin{alignat*}{2}
        &\min\bigl\{g(\hat r+Mv)+f(x)+\delta(s)\bigm\vert{}&&[C(\hat r+Mv);\veczero]+Ax+s=b,\,V(\hat r+Mv)\le w,\\
        &&&v\in\Z^p,\,l\le x\le u,\,x\in\Z^n,\,s\in\Z^m\bigr\},\\
        ={}&\min\bigl\{g(\hat r+Mv)+f(x)+\delta(s)\bigm\vert{}&&Ax+s=b+[-C\hat r-CMv;\veczero],\\
        &&&l\le x\le u,\,x\in\Z^n,\,s\in\Z^m,\,MVv\le w-V\hat r,\,v\in\Z^p\bigr\},
    \end{alignat*}
    which we can reformulate as
    \begin{equation}
        \label{eq:generic-fixed-phase-value-function-reformulation}
        \min\bigl\{g(\hat r+Mv)+h(-Cv)\bigm\vert MVv\le w-V\hat r,\,v\in\Z^p\bigr\}
    \end{equation}
    using the value function $h\colon\Z^d\to\R$ given by
    \begin{align*}
        h(z)=\min\bigl\{f(x)+\delta(s)\bigm\vert{}&[A,I][x;s]=(b+[-C\hat r;\veczero])+M[z;\veczero],\\
        &l\le x\le u,\,x\in\Z^n,\,s\in\Z^m\bigr\}.
    \end{align*}
    The restriction of $x$ to $l\le x\le u$ in the definition of $h$ can be moved into the objective by adding an indicator function $\xi$ satisfying $\xi(x)=0$ if and only if $l\le x\le u$ and $\xi(x)=\infty$ otherwise, which is separable convex. Hence, $h$ is of the form (\ref{eq:value-function-assumed-convex-extensibility}) for the separable convex objective function $[x;s]\mapsto f(x)+\delta(s)+\xi(x)$. By assumption, $h$ is convex extensible on $\Z^n\cap S$ for $S=\{-Cv:v\in\R^p\}$, which has dimension at most $p$. Now, it follows that the objective $v\mapsto g(\hat r+Mv)+h(-Cv)$ of (\ref{eq:generic-fixed-phase-value-function-reformulation}) is convex extensible. In addition, it is finite because $g,f$, and $\delta$ are finite, $\delta\ge0$, and the integral system $Ax+s=\hat b$, $l\le x\le u$ always has at least one feasible solution: $x=l$, $s=\hat b-Al$. Therefore, the objective is discrete convic and we can use the algorithm by Veselov et al.~\cite{DBLP:journals/dam/VeselovGZC20} to find an optimal solution $v$ to (\ref{eq:generic-fixed-phase-value-function-reformulation}). Since $v$ can be restricted to lie inside a ball of radius $\rho$, this requires $2^{\O(p^2\log p)}\log^{\O(1)}\rho$ arithmetic operations and $2^{\O(p^2\log p)}\log\rho$ evaluations of the constraints and comparisons of the objective of (\ref{eq:generic-fixed-phase-value-function-reformulation}).
    
    Evaluating the constraints $MVv\le w-V\hat r$ takes time $\O(\ell p)$. To implement the comparison oracle of the objective of (\ref{eq:generic-fixed-phase-value-function-reformulation}) we compute $\hat r+Mv$ and $-Cv$ in time $\O(dp)$ and compare $g(\hat r+Mv)+f(x)+\delta(s)$ for $[x;s]$ that satisfy $h(-Cv)=f(x)+\delta(s)$. Such $[x;s]$ is an optimal solution of the subproblem
    \[
        \min\bigl\{f(x)+\delta(s)\bigm\vert Ax+s=\hat b,\,l\le x\le u,\,x\in\Z^n,\,s\in\Z^m\bigr\}
    \]
    for $\hat b=(b+[-C\hat r;\veczero])+M[-Cv;0]$. Here, the constraints impose $\hat b-Al-s=A(x-l)$, which shows that we may impose the bounds $\|s-(\hat b-Al)\|_\infty\le\|A\|_\infty\|u-l\|_\infty$, i.e., we may restrict $\hat b-Al-(U\|u-l\|_\infty)\cdot\vecone\le s\le\hat b-Al+(U\|u-l\|_\infty)\cdot\vecone$ where $\vecone$ is the all-ones vector. Thus, after computing $Al$ in time $T_{\mathrm{mul}}$ and adjusting the variable bounds of $s$ in time $\O(m)$, an optimizer $(x,s)$ can be found by solving a problem of the form (\ref{eq:generic-second-stage-ip}) for variable domains with a length of at most $\sigma=(1+2U)\|u-l\|_\infty$, in time $T_{\mathrm{opt}}((1+2U)\|u-l\|_\infty)$. Note that $x=Al$, $s=\hat b-Al$ is a trivial initial feasible solution to the subproblem and can be provided to the optimization oracle.

    In this way, we can obtain an optimal solution $v$ to (\ref{eq:generic-fixed-phase-value-function-reformulation}) for each guess of $\hat r$, of which the best yields the value of $y$ in an optimal solution to (\ref{eq:feasible-generic-2-stage-ip}). This can be completed to an optimal solution $[y;x]$ in a negligible amount of time by solving a single additional subproblem of the form (\ref{eq:generic-second-stage-ip}).
\end{proof}

Note that in order to apply \cref{lemma:fixed-phase-value-function-reformulation}, the exact value of $M$ must be provided. For our applications, the computations in the proofs in \cref{sec:intro-periodic-convexity,appendix:primal-treedepth-scaled-idp} show that the value of $M$ can be computed within time $\O(M)$ when using a linear prime-number sieve to evaluate $\lcm\{1,2,\dots,N\}$~\cite{DBLP:journals/scp/Pritchard87}.

\subsection{Two-stage stochastic integer programs}

The algorithm for solving two-stage stochastic integer programs with constraint matrices of the form (\ref{eq:2-stage-matrix}) can now be obtained by applying the framework to a block-diagonal second-stage matrix $A$ while using the periodic convexity established in \cref{theorem:general-convex-extensibility}.

\theoremtwostagealgorithm*

\begin{proof}
    We apply \cref{lemma:fixed-phase-value-function-reformulation} to the constraint matrices $C\in\Z^{tn\times r}$ and $A\in\Z^{tn\times sn}$ given by
    \[
        C=\begin{bmatrix}
            C_1\\
            \vdots\\
            C_n
        \end{bmatrix},
        \quad\quad
        A=\begin{bmatrix}
            D_1&&\\
            &\ddots&\\
            &&D_n
        \end{bmatrix},
    \]
    while copying the right-hand side vector $b=[b_1;\dots;b_n]$, second-stage variable bounds $l=[l_1;\dots;l_n]$, $u=[u_1;\dots;u_n]$, and encoding $l_0\le x_0\le u_0$ in $Vx_0\le w$ with $\ell=2r$ constraints. The objective remains $g(x_0)+f'(x_1,\dots,x_n)=f(x_0,x_1,\dots,x_n)$.
    
    Since the value function in (\ref{eq:value-function-assumed-convex-extensibility}) is the sum of value functions for constraint matrices of the form $[D_i,I]\in\{-\Delta,-\Delta+1,\dots,\Delta\}^{t\times(s+t)}$, \cref{theorem:general-convex-extensibility} guarantees that (\ref{eq:value-function-assumed-convex-extensibility}) is convex extensible for $M=2^{\O((\sqrt t\Delta)^t)}$. We can compute a product $Ax$ in time $T_{\mathrm{mul}}=\O(nst)$ and can estimate $\|A\|_\infty\le s\Delta=:U$. An optimal solution to a subproblem of the form (\ref{eq:generic-second-stage-ip}) can be found by independently solving the $n$ subproblems corresponding to the diagonal $t\times(s+t)$ blocks $[D_i,I]$. Each such separable convex integer program can be solved using the dual treedepth algorithm from~\cite{DBLP:conf/ipco/HunkenschroderKLV25}, which runs in time $(t\Delta)^{\O(t^2)}\cdot s\log(s)\cdot\log\sigma$. This yields $T_{\mathrm{opt}}(\sigma)=n\cdot(t\Delta)^{\O(t^2)}\cdot s\log(s)\cdot\log\sigma$. Note that the feasible region of the global variables is contained in a ball of radius $\|u_0-l_0\|_\infty$, meaning that we can bound $\log\rho\le L$.

    From \cref{lemma:fixed-phase-value-function-reformulation}, we obtain the running time
    \begin{alignat*}{2}
        &\mathrlap{M^r\cdot2^{\O(r^2\log r)}\cdot\Bigl(\log^{\O(1)}(\rho)+\log(\rho)\cdot\bigl(T_{\mathrm{opt}}((1+2U)\|u-l\|_\infty)+T_{\mathrm{mul}}+tn+2r\bigr)\Bigr)}&&\\
        ={}&\bigl(2^{\O((\sqrt t\Delta)^t)}\bigr)^r\cdot{}&&2^{\O(r^2\log r)}\cdot\Bigl(L^{\O(1)}+L\cdot\bigl(n\cdot(t\Delta)^{\O(t^2)}\\
        &&&\cdot s\log(s)\cdot\log((1+2s\Delta)\|u-l\|_\infty)+\O(nst)+tn+2r\bigr)\Bigr)\\
        ={}&\mathrlap{2^{\O((\sqrt t\Delta)^t\cdot r)}\cdot2^{\O(r^2\log r)}\cdot\bigl(L^{\O(1)}+L\cdot n\cdot(t\Delta)^{\O(t^2)}\cdot s\log(s)\cdot(\log(s)+\log(\Delta)+L)\bigr)}&&\\
        ={}&\mathrlap{2^{\O((\sqrt t\Delta)^t\cdot r+r^2\log r)}\cdot s\log^{\O(1)}(s)\cdot n\cdot L^{\O(1)}.}&&\tag*{\qedhere}
    \end{alignat*}
\end{proof}
In the case of two-stage stochastic integer programs with second-stage constraint matrices $D_i$ that have $s=\Omega(t)$, this improves upon the triply exponential algorithm by Eisenbrand and Rothvoss~\cite{DBLP:conf/soda/EisenbrandR26}. We note that the parameterization of two-stage stochastic integer programs by the number $t$ of constraints of each brick instead of by the number $s$ of variables of each brick is unusual. However, in the arguments accumulating in \cref{theorem:general-convex-extensibility}, redundant linearly dependent constraints in $Ax^i=r+Mz^i$ may be discarded, resulting in effectively $t\le s$ and a parametric running time dependency of $2^{\O((\sqrt s\Delta)^s)\cdot r+r^2\log r}$. This reduction may not be possible when the constraints are formulated as inequalities, which raises the question of whether value functions of the form $b\mapsto\min\{f(x)\ \vert\ Ax\le b,\,x\in\Z^n\}$ become convex extensible on $r+M\Z^m$ for some $M$ that only depends on the number of variables $n$ and coefficient size $\Delta$ of $A$. Here, the results of Eisenbrand and Rothvoss~\cite{DBLP:conf/soda/EisenbrandR26} imply that $M$ needs to be at most triply exponential in order for this to work for linear objectives $f$.

\subsection{Integer programs with dual block-structure}

We can also use a value function reformulation to derive an $n$-fold IP algorithm supporting large coefficients in the global constraints if they are uniform, i.e., $B_1,\dots,B_n$ in (\ref{eq:n-fold-matrix}) are all equal to some matrix $B\in\Z^{r\times t}$. Here, we make use of the following simple observation: an IP with constraints (\ref{eq:uniform-n-fold-matrix}) is equivalent to one with constraints (\ref{eq:reformulated-uniform-n-fold-matrix}) by introducing new variables $y\in\Z^t$ that aggregate $y=x_1+\dots+x_n$. Both formulations are shown below.
\[
    \begin{array}{ccc}
         & \multirow{2}{*}{$\quad\begin{bmatrix}
            B&&&\\
            -I&I&\cdots&I\\
            &D_1&&\\
            &&\ddots&\\
            &&&D_n
        \end{bmatrix}\begin{bmatrix}
            y\\
            x_1\\
            \vdots\\
            x_n
        \end{bmatrix}=\begin{bmatrix}
            b_0\\
            \veczero\\
            b_1\\
            \vdots\\
            b_n
        \end{bmatrix}$}\\
        \begin{bmatrix}
            B&\cdots&B\\
            D_1&&\\
            &\ddots&\\
            &&D_n
        \end{bmatrix}\begin{bmatrix}
            x_1\\
            \vdots\\
            x_n
        \end{bmatrix}=\begin{bmatrix}
            b_0\\
            b_1\\
            \vdots\\
            b_n
        \end{bmatrix}\quad & & \\
        \manualeqlabel{eq:uniform-n-fold-matrix}\quad & \quad\manualeqlabel{eq:reformulated-uniform-n-fold-matrix}
    \end{array}
\]
Hence, we can construct an equivalent value function reformulation where the second stage is an $n$-fold IP without large coefficients. Here, periodic convexity of the objective function of $y$ can be established using \cref{corollary:n-fold-scaled-idp,lemma:fixed-dimension-periodic-convexity}. We note that Chen, Chen, and Zhang~\cite{DBLP:journals/disopt/ChenCZ22} use the conversion from (\ref{eq:uniform-n-fold-matrix}) to (\ref{eq:reformulated-uniform-n-fold-matrix}) to rewrite a $4$-block IP when $D_i=[1,\dots,1]\in\Z^{1\times t}$, in which case the second-stage problem simplifies significantly by becoming totally unimodular.

\theoremnfoldalgorithm*

\begin{proof}
    We apply \cref{lemma:fixed-phase-value-function-reformulation} to the $n$-fold matrix
    \begin{equation}
        \label{eq:transformed-n-fold-matrix}
        A=\begin{bmatrix}
            I&\cdots&I\\
            D_1&&\\
            &\ddots&\\
            &&D_n
        \end{bmatrix},
    \end{equation}
    the linking matrix $C=-I$, $d=t$, right-hand side vector $b=[\veczero;b_1;\dots;b_n]$, variable bounds $l=[l_1;\dots;l_n]$, $u=[u_1;\dots;u_n]$, $g=y\mapsto0$, and the original objective $f$ as second-stage objective. Note that any feasible solution $x$ in (\ref{eq:n-fold-ip}) satisfies $l_1+\dots+l_n\le y\le u_1+\dots+u_n$ for $y=x_1+\dots+x_n$. We can encode these bounds on $y$ and the global constraints $Bx_1+\dots+Bx_n=b_0=By$ in $Vy\le w$ with $2t+2r$ constraints. Note that $\{y\in\R^t:Vy\le w\}$ is contained in a ball of radius $\rho=\|\sum_{i\in[n]}u_i-\sum_{i\in[n]}l_i\|_\infty\le n\max_{i\in[n]}\|u_i-l_i\|_\infty$ and thus that $\log\rho\le\log(n)+L$.

    We now compute a period $M$ for which (\ref{eq:value-function-assumed-convex-extensibility}) is convex extensible by applying \cref{corollary:n-fold-scaled-idp,lemma:fixed-dimension-periodic-convexity} to $d=t$ and the matrix $[A;I]$. We may observe that $[A;I]$ and the further extended $A'$ in the statement of \cref{lemma:fixed-dimension-periodic-convexity} are both $n$-fold matrices with $t$ global constraints, $s$ local constraints and coefficients bounded by $\Delta$. Hence, \cref{corollary:n-fold-scaled-idp} shows that $[A,I]$ has the IDP after an $N$-dilation for $N=2^{\O((t+s)\Delta)^{t(t+s)}}$. Additionally, $A'$ has an $\ell_\infty$-Graver complexity bounded by $G=(ts\Delta)^{\O(ts)}$ as a result of the Graver complexity bound for $n$-fold matrices from~\cite{DBLP:conf/icalp/EisenbrandHK18}. Now, \cref{lemma:fixed-dimension-periodic-convexity} shows that we obtain our desired periodic convexity of (\ref{eq:value-function-assumed-convex-extensibility}) for a period $M=\O(G)^{t(t+1)}\cdot N=((ts\Delta)^{\O(ts)})^{t(t+1)}\cdot2^{\O((t+s)\Delta)^{t(t+s)}}=2^{\O((t+s)\Delta)^{t(t+s)}}$.

    With respect to the remaining conditions in \cref{lemma:fixed-phase-value-function-reformulation}, we can estimate $\|A\|_\infty\le n+s\Delta=:U$, compute a product $Ax$ in time $T_{\mathrm{mul}}=\O(nst)$, and solve (\ref{eq:generic-second-stage-ip}) using the dual treedepth algorithm from~\cite{DBLP:conf/ipco/HunkenschroderKLV25}, which runs in time $T_{\mathrm{opt}}(\sigma)=(ts\Delta)^{\O(ts(t+s))}\cdot tn\log(tn)\cdot\log\sigma$.

    \cref{lemma:fixed-phase-value-function-reformulation} shows that we can solve (\ref{eq:n-fold-ip}) in time
    \begin{alignat*}{2}
        &M^t\cdot{}&&2^{\O(t^2\log t)}\cdot\Bigl(\log^{\O(1)}(\rho)+\log(\rho)\\
        &&&\cdot\bigl(T_{\mathrm{opt}}((1+2U)\|u-l\|_\infty)+T_{\mathrm{mul}}+(t+sn)+(2t+2r)\bigr)\Bigr)\\
        ={}&\mathrlap{\bigl(2^{\O((t+s)\Delta)^{t(t+s)}}\bigr)^t\cdot2^{\O(t^2\log t)}\cdot\Bigl((\log(n)+L)^{\O(1)}+(\log(n)+L)\cdot\bigl((ts\Delta)^{\O(ts(t+s))}}&&\\
        &&&\cdot tn\log(tn)\cdot\log((1+2(n+t\Delta))\|u-l\|_\infty)+\O(nst)+sn+3t+2r\bigr)\Bigr)\\
        ={}&\mathrlap{2^{\O((t+s)\Delta)^{t(t+s)}}\cdot\Bigl(\log^{\O(1)}(n)\cdot L^{\O(1)}\cdot\bigl((ts\Delta)^{\O(ts(t+s))}}&&\\
        &&&\cdot tn\log(tn)\cdot(\log(n+t\Delta)+L)+r\bigr)\Bigr)\\
        ={}&\mathrlap{2^{\O((t+s)\Delta)^{t(t+s)}}\cdot r\cdot n\log^{\O(1)}(n)\cdot L^{\O(1)}.}&&\tag*{\qedhere}
    \end{alignat*}
\end{proof}

A useful consequence of reducing to $n$-fold IP with small coefficients is that the fixed-parameter tractability in the presence of large coefficients extends to the setting where the $n$-fold IP is given in a high-multiplicity encoding: Knop et al.~\cite{DBLP:journals/mp/KnopKLMO23} provide an algorithm that efficiently minimizes compactly encoded $n$-fold IPs in the case that all coefficients are small.

In \cref{theorem:n-fold-algorithm}, we parameterize by the number of columns of $B$ and assume uniformity of the global constraints to match the parameterization used in~\cite{DBLP:journals/theoretics/CslovjecsekKLPP25}. However, we can also treat the more common parameterization by $r$ and $s$ from the literature~\cite{DBLP:conf/soda/CslovjecsekEHRW21,DBLP:conf/icalp/EisenbrandHK18}, which will be a special case of the following discussion.

Using rewriting techniques similar to those used in the conversion from (\ref{eq:uniform-n-fold-matrix}) to (\ref{eq:reformulated-uniform-n-fold-matrix}), we show that IPs of bounded dual treedepth with small coefficients are still efficiently solvable under the presence of a fixed number of additional constraints containing a fixed number of, possibly large, coefficients. This yields \cref{theorem:dual-treedepth-algorithm}, where we rely on \cref{corollary:dual-treedepth-scaled-idp,lemma:fixed-dimension-periodic-convexity} and the Graver complexity bound from~\cite{DBLP:journals/toct/KnopPW20}.

\begin{theorem}
    \label{theorem:dual-treedepth-algorithm}
    Let $\{k_1,\dots,k_\ell\}=K\subseteq\Z$, $B\in K^{r\times n}$, $D\in\{-\Delta,-\Delta+1,\dots,\Delta\}^{m\times n}$, $l,u\in\Z^n$, and $b^B\in\Z^r$, $b^D\in\Z^m$. Assume that $D$ has dual treedepth $d$. Let $f\colon\R^n\to\R$ be separable convex and accessible through a comparison oracle on $\Z^n$. An optimal solution to the integer program
    \begin{equation}
        \label{eq:treedepth-ip}
        \min\bigl\{f(x)\bigm\vert Bx=b^B,\,Dx=b^D,\,l\le x\le u,\,x\in\Z^n\bigr\}
    \end{equation}
    can be found in time $2^{\O(\Delta)^{2^{\ell r+d}}}\cdot n\log^{\O(1)}(n)\cdot L^{\O(1)}$, where $L=\log\|u-l\|_\infty$.
\end{theorem}
Note that the arbitrary coefficients $K$ must necessarily be restricted to a bounded number of rows $r$, because $n$-fold IP becomes NP-hard even when only one single large coefficient is distributed arbitrarily throughout the $n$-fold constraint structure~\cite{DBLP:journals/disopt/ChenCZ22}.

Given the similarity between the proof of \cref{theorem:dual-treedepth-algorithm} and the proof of \cref{theorem:n-fold-algorithm}, which both involve tedious substitutions and calculations, we postpone the proof of \cref{theorem:dual-treedepth-algorithm} to \cref{appendix:algorithms}.

We note that an earlier version of this work~\cite{ligthart2026valuefunctionsseparableconvexv1} contained variants of \cref{theorem:n-fold-algorithm,theorem:dual-treedepth-algorithm} with slightly better parametric dependences of $2^{(\O(s\Delta)^s\cdot t)^t}$ and $2^{\O(\Delta)^{2^d\cdot\ell r}\cdot(\ell r)^{\ell r}}$ respectively. These can be obtained by working with a restricted notion of the IDP, exploiting the fact that only the top $d$ coordinates of the right-hand side in (\ref{eq:value-function-assumed-convex-extensibility}) are varying. This discussion has been omitted here in favor of the more general \cref{lemma:idp-dual-glueing}.

\subsection{Other applications}

The proofs of \cref{theorem:n-fold-algorithm,theorem:dual-treedepth-algorithm} rely on the value function arguments being restricted to a subspace of the form $\R^d\times\{b\}$ to bound the period with which periodic convexity occurs through \cref{lemma:fixed-dimension-periodic-convexity}. If we instead use \cref{proposition:idp-iff-convex-extensibility-along-line} to bound the period for which value functions of $n$-fold IPs become periodically convex when evaluated along $1$-dimensional affine subspaces, we can obtain a straightforward algorithm to solve a restricted subclass of $4$-block IPs, with a constraint matrix of the form (\ref{eq:4-block-matrix}). More specifically a combination of \cref{corollary:n-fold-scaled-idp}, \cref{proposition:idp-iff-convex-extensibility-along-line}, and \cref{lemma:fixed-phase-value-function-reformulation} yields \cref{theorem:restricted-4-block-algorithm}, an algorithm with a near-linear running time dependence on $n$ when $A_0,C_1,\dots,C_n$ in (\ref{eq:4-block-matrix}) have exactly $1$ column, i.e., the $4$-block IP has only a single global variable. We note that $4$-block integer programs are known to be FPT if the number of global variables is fixed~\cite{DBLP:conf/ipco/LassotaL26,DBLP:journals/mp/OertelPW24}. Despite this, for a single global variable, all previously known algorithms have a running time that is at least quadratic in $n$~\cite{DBLP:journals/mp/ChenCZ24,DBLP:journals/mp/HemmeckeKW14,DBLP:conf/ipco/LassotaL26,DBLP:journals/mp/OertelPW24}. Chen, Chen, and Zhang~\cite{DBLP:journals/mp/ChenCZ24} show that some scheduling and delivery problems can be captured in such $4$-block IPs with $1$ global variable.\footnote{The authors of~\cite{DBLP:journals/mp/ChenCZ24} provide algorithms for a slightly larger class of \emph{almost combinatorial $4$-block IPs} that have constraint matrices of the form (\ref{eq:4-block-matrix}) where all coefficients are bounded, $C_1=\dots=C_n$ has rank $1$, but $A_0$ can have arbitrary rank.} In this way, \cref{theorem:restricted-4-block-algorithm} yields faster parameterized algorithms for those applications.

\begin{theorem}
    \label{theorem:restricted-4-block-algorithm}
    Let $a\in\Z^r$, $c_i\in\Z^s$, $B_i\in\{-\Delta,-\Delta+1,\dots,\Delta\}^{r\times t}$, $D_i\in\{-\Delta,-\Delta+1,\dots,\Delta\}^{s\times t}$, $l_0,u_0\in\Z$, $b_0\in\Z^r$, $l_i,u_i\in\Z^t$, and $b_i\in\Z^s$ for $i\in[n]$. Let $f\colon\R^{1+tn}\to\R$ be separable convex and accessible through a comparison oracle on $\Z^{1+tn}$. An optimal solution to the integer program
    \begin{alignat*}{2}
        \min\bigl\{f(x_0,x_1,\dots,x_n)\bigm\vert{}&\mathrlap{ax_0+B_1x_1+\dots+B_nx_n=b_0,}\\
        &c_ix_0+D_ix_i=b_i,\,{}&&l_i\le x_i\le u_i,\,x_i\in\Z^t,\,i\in[n]\\
        &&&l_0\le x_0\le u_0,\,x_0\in\Z\bigr\}
    \end{alignat*}
    can be found in time $2^{\O((r+s)\Delta)^{r(r+s)}}\cdot tn\log^{\O(1)}(tn)\cdot L^{\O(1)}$, where $L=\log\max_{i\in\{0,1,\dots,n\}}\|u_i-l_i\|_\infty$.
\end{theorem}
The proof is postponed to \cref{appendix:algorithms}. Note that the double exponential running time lower bound from~\cite{DBLP:journals/mp/JansenKL23} applies to \cref{theorem:restricted-4-block-algorithm}, since it contains the two-stage stochastic integer programming problem with a single global variable.

By using both the bounds from \cref{corollary:dual-treedepth-scaled-idp} as well as \cref{corollary:primal-treedepth-scaled-idp} in conjunction with \cref{lemma:fixed-dimension-periodic-convexity}, we can apply the framework from \cref{lemma:fixed-phase-value-function-reformulation} to partially answer an open question from Brand, Kouteck\'y and Ordyniak~\cite{DBLP:conf/aaai/BrandKO21} in the setting of mixed integer linear programs. The authors from~\cite{DBLP:conf/aaai/BrandKO21} observe that in order for a class of mixed integer linear programs to be tractable, it is not sufficient to only restrict the structure of the constraint submatrix induced by the set of integer variables: identity matrices can be used to enforce integrality constraints on arbitrary continuous variables. The authors ask whether it is sufficient to limit the number of constraints through which arbitrary continuous variables interact with a tractable integer program. If this latter IP has bounded treedepth and coefficients, \cref{theorem:restricted-interaction-milp} answers this affirmatively.

\begin{theorem}
    \label{theorem:restricted-interaction-milp}
    Let $A_\Z\in\{-\Delta,-\Delta+1,\dots,\Delta\}^{m\times n_\Z}$ have primal or dual treedepth bounded by $d$. Let $c_\Z\in\R^{n_\Z}$, $A_\R\in\R^{m\times n_\R}$, $c_\R\in\R^{n_\R}$ and $b\in\R^m$. Assume that there are at most $k$ indices $i\in[m]$ for which the $i$-th row of $A_\R$ and $A_\Z$ are simultaneously nonzero. An optimal solution to the mixed integer linear program
    \begin{equation}
        \label{eq:milp}
        \min\bigl\{c_\R^\top x_\R+c_\Z^\top x_\Z\bigm\vert A_\R x_\R+A_\Z x_\Z=b,x_\R\in\R_{\ge0}^{n_\R},\,x_\Z\in\Z_{\ge0}^{n_\Z}\bigr\}
    \end{equation}
    can be found in time $f(\Delta,d,k)\cdot|I|^{\O(1)}$, where $|I|$ is the encoding length of (\ref{eq:milp}) and $f$ is a computable function.
\end{theorem}
The proof is postponed to \cref{appendix:algorithms}.

%% file: sections/future-directions.tex
\section{Future directions}

We have shown that the value function of an integer program with separable convex objectives is periodically convex. That is, $b\mapsto\min\{f(x)\ \vert\ Ax=b,\,x\in\Z^n\}$ is convex extensible on any lattice translate $r+M\Z^m$ for any $r\in\Z^m$ and some number $M$ depending only on $A$. We have also given improved parameterized bounds when $A$ is block-structured and the variations of $b$ are restricted to a bounded dimensional subspace of the form $\R^d\times\{\veczero\}$. Finally, we have used periodic convexity in combination with value function reformulations of block-structured IPs to derive FPT algorithms for block-structured IPs that can simultaneously handle large entries in the global part of the constraint matrix and separable convex objective functions.

Despite improving significantly on the triply exponential running time of the $n$-fold IP algorithm by Cslovjecsek et al.~\cite{DBLP:journals/theoretics/CslovjecsekKLPP25}, the algorithms of \cref{theorem:n-fold-algorithm} and \cref{theorem:dual-treedepth-algorithm} still have running times that contain an exponent tower of one height higher than their lower bounds~\cite{DBLP:journals/toct/KnopPW20} and their small-coefficient algorithm counterparts~\cite{DBLP:conf/soda/CslovjecsekEHRW21,DBLP:conf/icalp/EisenbrandHK18,DBLP:journals/mor/EisenbrandHKKLO25,eisenbrand2022algorithmictheoryintegerprogramming,DBLP:conf/ipco/HunkenschroderKLV25}. A natural question is whether this exponential time complexity increase can be avoided. If this is possible, a different algorithmic approach must be used. See \cref{appendix:lower-bounds}.

It is remarkable that the bound on $M$ in \cref{theorem:general-convex-extensibility} is dominated by the dilation needed to establish the IDP from \cref{lemma:general-scaled-idp}, i.e., periodic convexity on $1$-dimensional subspaces. It would be interesting to determine whether there is a natural class of IPs for which the IDP is attained after an $M$-dilation for an $M$ that is much smaller than the smallest period for which periodic convexity occurs on the complete space of right-hand sides.

An important open problem that remains is the unknown parameterized complexity of the 4-block integer linear programming problem. If the periodic convexity can be extended from $1$-dimensional subspaces to arbitrary affine subspaces (of any fixed dimension), this would yield a straightforward FPT algorithm for the general 4-block integer programming problem with any parameterized number of global variables. To show such periodic convexity, different techniques from those used in \cref{lemma:reduction-to-bounded-box,claim:vector-rearrangement-with-insertion} are necessary because these crucially rely on the variations coming a subspace of the form $\R^d\times\{\veczero\}$, where $d$ is bounded. Given the applications to value function reformulations, overcoming these dimension-related limitations for any sort of structured constraint matrix with many rows would be of interest.

%% file: sections/appendix-lower-bounds.tex
\section{Lower bounds}
\label{appendix:lower-bounds}

Here we argue that the bounds from \cref{theorem:general-convex-extensibility,corollary:dual-treedepth-scaled-idp} are asymptotically (almost) tight. To do so, we make use of a well-known matrix construction from the literature. See, e.g.~\cite{DBLP:journals/mp/Klein22,DBLP:journals/toct/KnopPW20}. Let $m,\Delta\in\Z_{\ge1}$ and let $N$ be a positive integer that is at most $U=\Delta^m-1$. Then define the $m\times m$ matrix
\begin{equation}
    \label{eq:encoding-matrix}
    A_N=\left[\begin{array}{*{5}{w{c}{36pt}}}
        a_0(N)&a_1(N)&\cdots&a_{m-2}(N)&a_{m-1}(N)\\
        \Delta&-1&&&\\
        &\Delta&\ddots&&\\
        &&\ddots&-1&\\
        &&&\Delta&-1
    \end{array}\right],
\end{equation}
where $a_0(N),\dots,a_{m-1}(N)$ is the base $\Delta$ encoding of $N=\sum_{i\in\{0,1,2,\dots,m-1\}}a_i(N)\cdot\Delta^i$. Using coefficients from $\{-1,0,1,2,\dots,\Delta\}$, the matrix $A_N$ effectively encodes the larger coefficient $N$ in the following way: for all $x_2,x_1\ge0$ it holds that $x_2=Nx_1$ if and only if there exists $y\in\R_{\ge0}^{m-1}$ such that $[A_N,-e_1][x_1;y;x_2]=\veczero$. Here, $e_i$ is the $i$-th standard basis column vector.

\begin{proposition}
    \label{proposition:adaptive-lower-bound}
    Any positive integer $M$ satisfying the conditions of \cref{lemma:general-scaled-idp} or \cref{theorem:general-convex-extensibility} must be at least $M\ge2^{\Omega(\Delta^m)}$.
\end{proposition}

\begin{proof}
    If the polyhedron $\{x\in\R_{\ge0}^m:A_Nx=M[1;\veczero]\}$ has the IDP, it must be integral. Therefore, its only point $x\in\R_{\ge0}^m$, which satisfies $x_1=M/N$, must be integral. Applying this argument for all $N\in\{1,2,\dots,U\}$ shows that $\lcm\{1,2,\dots,U\}=2^{\Omega(U)}=2^{\Omega(\Delta^m)}$~\cite{Nair01021982} divides any $M$ that satisfies the conditions of \cref{lemma:general-scaled-idp}. As a result of \cref{proposition:idp-iff-convex-extensibility-along-line}, this lower bound also applies to \cref{theorem:general-convex-extensibility}.
\end{proof}

Note that the lower bound from \cref{proposition:adaptive-lower-bound} relies on the fact that $M$ is not allowed to be chosen adaptively based on the matrix $A\in\{-\Delta,-\Delta+1,\dots,\Delta\}^{m\times n}$ or the objective function. This raises the question of whether the bound on the period $M$ can be improved if one only demands periodic convexity of a single specific IP value function. This is of interest in the setting where one wishes to solve an integer linear program (ILP) of the form $\min\{c^\top x\ \vert\ Ax=b,\,l\le x\le u,\,x\in\Z^n\}$, where all, except for $k$, columns of $A\in\Z^{m\times n}$ have coefficients bounded by $\Delta$ in absolute value, which is considered by Dadush et al.~\cite{DBLP:conf/ipco/DadushLRV23}. The authors solve such ILP in time $n^m\cdot(m\Delta)^{\O(m^2)}\cdot k^{\O(k)}$ times a polynomial in the encoding length of the instance. A natural question is whether an FPT algorithm, parameterized by $(m,k,\Delta)$, exists that has a running time that avoids the slice-wise polynomial dependency of $n^m$, but maintains the moderate dependency on $\Delta$. This ILP problem fits into the framework of \cref{lemma:fixed-phase-value-function-reformulation} and can be solved in FPT time, depending exponentially on $\Delta$, using the bound from \cref{theorem:general-convex-extensibility}. In such application of \cref{lemma:fixed-phase-value-function-reformulation}, one only needs the periodic convexity of the value function of the ILP restricted to the columns with small coefficients. Unfortunately, \cref{proposition:non-adaptive-super-polynomial-lower-bound} shows that the modulus to obtain a convex program on the $k$ integer variables, depends superpolynomially on $\Delta$, even when $m=2$.

\begin{proposition}
    \label{proposition:non-adaptive-super-polynomial-lower-bound}
    For any $\Delta\in\Z_{\ge0}$, there exists a constraint matrix $A\in\{-\Delta,-\Delta+1,\dots,\Delta\}^{2\times n}$ and objective vector $c\in\Z^n$, so that all positive integers $M$ for which $b\mapsto\min\{c^\top x\ \vert\ Ax=b,\,x\in\Z_{\ge0}^n\}$ is convex extensible along any line on $M\Z^2$ must satisfy $M\ge2^{\Omega(\sqrt\Delta)}$.
\end{proposition}

\begin{proof}        
    Let $\Delta\in\Z_{\ge1}$ be fixed. Let $\Gamma=\lfloor\sqrt\Delta\rfloor$ and set $L=\lcm\{1,\dots,\Gamma\}=2^{\Omega(\Gamma)}=2^{\Omega(\sqrt\Delta)}$~\cite{Nair01021982}. We construct the $2\times\Gamma$ matrix $A$ with columns $A_i=[i;i^2]$ and $c_i=L^{\Gamma-i}$ for $i\in[\Gamma]$. Let $M$ satisfy the condition from \cref{proposition:non-adaptive-super-polynomial-lower-bound}. For the sake of deriving a contradiction, assume that $M<L$. For each $i\in[\Gamma]$, we evaluate the value function $h(b)=\min\{c^\top x\ \vert\ Ax=b,\,x\in\Z_{\ge0}^n\}$ at the points $h(\veczero)=0$ and $b=[Mi;Mi^2]$. Note that $M\cdot e_i$, is a feasible solution for the latter right-hand side and attains objective value $Mc_i$. If there exists another solution, as a result of the unique slopes of the lines defined by the columns of the constraint matrix, such solution must assign a strictly positive value to some $x_j$ for $j<i$. Hence, the objective value of such alternative solution is at least $c_j\ge c_{i-1}=L^{\Gamma-i+1}=Lc_i>Mc_i$, which shows that $h([Mi;Mi^2])=Mc_i$. Now, convex extensibility along the line from $\veczero$ to $M[i;i^2]$ on $M\Z^2$ shows that $h([M;Mi])\le\tfrac{i-1}ih(\veczero)\cdot\tfrac1ih(i\cdot[M;Mi])=\tfrac Mic_i$. Analogously to the previous argument, any solution to the integer program with right-hand side $[M;Mi]$ that uses any variable other than $x_i$ has objective value at least $c_{i-1}>Mc_i\ge\tfrac Mic_i$, which cannot be optimal. Therefore, the optimal solution must solely use variable $x_i$ and satisfy $[i;i^2]\cdot x_i=[M,Mi]$, which shows that $M$ must be divisible by $i$. Since $i\in[\Gamma]$ was arbitrary, $L$ divides $M$, which shows that $M\ge L$.
\end{proof}
Note that from the proof it is clear that the same lower bound applies for finite variable upper bounds. We leave it open whether $M$ must also grow doubly exponentially in $m$.

We can combine the constructions from \cref{proposition:adaptive-lower-bound,proposition:non-adaptive-super-polynomial-lower-bound} to obtain a lower bound parameterized by dual treedepth that (almost) matches \cref{corollary:dual-treedepth-scaled-idp}.

\begin{proposition}
    \label{proposition:non-adaptive-dual-treedepth-lower-bound}
    For any $d,\Delta\in\Z_{\ge0}$, there exists a constraint matrix $A\in\{-\Delta,-\Delta+1,\dots,\Delta\}^{m\times n}$ with dual treedepth at most $d$ and objective vector $c\in\Z^n$, so that all positive integers $M$ for which $b\mapsto\min\{c^\top x\ \vert\ Ax=b,\,x\in\Z_{\ge0}^n\}$ is convex extensible along any line on $M\Z^m$ must satisfy $M\ge2^{\Delta^{\Omega(2^d)}}$.
\end{proposition}

\begin{proof}
    Set $m'=2^{d-5}$ and $\Gamma=\Delta^{m'}-1$. We can encode the system in the proof of \cref{proposition:non-adaptive-super-polynomial-lower-bound} in
    \begin{equation}
        \label{eq:non-adaptive-dual-treedepth-construction}
        \renewcommand{\arraystretch}{1.1}
        \left[\begin{array}{c c c;{0.5pt/1pt}c;{0.5pt/1pt}c c c}
        
            &e_1^\top&&\cdots&&e_1^\top&\\
            &&1&\cdots&&&1\\
            \cdashline{1-7}[0.5pt/1pt]
            A_1&-e_1e_1^\top&&&&&\\
            &A_1&-e_1^\top&&&&\\
            \cdashline{1-7}[0.5pt/1pt]
            &&&\ddots&&&\\
            \cdashline{1-7}[0.5pt/1pt]
            &&&&A_\Gamma&-e_1e_1^\top&\\
            &&&&&A_\Gamma&-e_1^\top\\
        \end{array}\right]x'=\begin{bmatrix}
            b_1\\
            b_2\\
            \veczero\\
            \veczero\\
            \vdots\\
            \veczero\\
            \veczero
        \end{bmatrix},
    \end{equation}
    using the matrices $A_N\in\{-\Delta,-\Delta+1,\dots,\Delta\}^{m'\times m'}$ from (\ref{eq:encoding-matrix}). For any $i\in[\Gamma]$, the variable $x_i$ from the construction in the proof of \cref{proposition:non-adaptive-super-polynomial-lower-bound} is equivalent to the variable $x_{1+(i-1)\cdot(2m'+1)}'$ in (\ref{eq:non-adaptive-dual-treedepth-construction}). For this reason, we set the $(1+(i-1)\cdot(2m'+1))$-th objective coefficient to $(\lcm\{1,\dots,\Gamma\})^{\Gamma-i}$ and the rest to zero. In this way, the value function of the optimization problem under the constraints (\ref{eq:non-adaptive-dual-treedepth-construction}), restricted to $[b_1;b_2]$ is identical to the one considered in \cref{proposition:non-adaptive-super-polynomial-lower-bound}. Therefore, any $M$ satisfying the conditions in \cref{proposition:non-adaptive-dual-treedepth-lower-bound} must satisfy $M\ge2^{\Omega(\Gamma)}=2^{\Omega(\Delta^{m'}-1)}\ge2^{\Delta^{\Omega(2^d)}}$.

    It is left to verify that the constraint matrix of (\ref{eq:non-adaptive-dual-treedepth-construction}) has dual treedepth at most $d$. First, note that the constraint matrix of (\ref{eq:non-adaptive-dual-treedepth-construction}) is of the form (\ref{eq:n-fold-matrix}) with $2$ global constraints and diagonal blocks $D_i$ of dimension $2m'\times(2m'+1)$. Each $D_i$ can be brought into form (\ref{eq:n-fold-matrix}) by identifying the top rows of both appearances of $A_i$ as global constraints and having the remaining $m'-1$ rows of each $A_i$ end up in the diagonal blocks. Both of these $2$-diagonal blocks have a dual treedepth of at most $\lceil\log m'\rceil$ (See, e.g.~\cite{DBLP:journals/mp/KnopKLMO23}). Therefore, the dual treedepth of constraint matrix of (\ref{eq:non-adaptive-dual-treedepth-construction}) is at most $2+(2+\lceil\log m'\rceil)\le5+\log m'=d$.
\end{proof}
Note that the constraint matrix in (\ref{eq:non-adaptive-dual-treedepth-construction}) is an $n$-fold constraint matrix with block width $t=2m'+1$, which can appear as a possible submatrix of (\ref{eq:transformed-n-fold-matrix}) in the proof of \cref{theorem:n-fold-algorithm}. This shows that the value of $M$ in that application of \cref{lemma:fixed-phase-value-function-reformulation} must be at least $2^{\Omega(\Delta^{m'}-1)}=2^{\Omega(\Delta^t)}$. Therefore, the used algorithmic approach to solve $n$-fold IPs with large global coefficients is limited to running times that have a doubly exponential dependence on the block dimensions.

%% file: sections/appendix-primal-treedepth.tex
\section{A bound on the dilation that establishes the IDP for matrices with bounded primal treedepth}
\label{appendix:primal-treedepth-scaled-idp}

In this section we prove \cref{lemma:idp-primal-glueing}. In the proof, we apply the integer decomposition property for each block $[C_i,D_i]$ independently, balance the resulting subsolutions with \cref{lemma:balancing}, and aggregate the projections of these subsolutions on the first-stage variables across all blocks using the following variant of the vector rearrangement lemmata from~\cite{DBLP:conf/esa/CslovjecsekEPVW21,DBLP:journals/mp/Klein22}:

\begin{lemma}
    \label{lemma:klein-lemma-variant}
    Let $m,\Delta\in\Z_{\ge0}$, and let $A\in\Z^{m\times n}$ be the matrix having precisely all vectors from $\{-\Delta,-\Delta+1,\dots,\Delta\}^m$ as columns. There exists a positive integer $N=2^{(m\Delta)^{\O(m^2)}}$ so that the following holds: let $k\in\Z_{\ge0}$ and $b\in\Z^m$. Then there exist $b^1,\dots,b^k\in\Z^m$ so that any $x\in\Z_{\ge0}^n$ satisfying $\|x\|_1=kN$ and $Ax=b$, can be decomposed into $x=x^1+\dots+x^k$ with $x^i\in\Z_{\ge0}^n$, $\|x^i\|_1=N$, and $Ax^i=b^i$ for $i\in[k]$.
\end{lemma}

We first prove \cref{lemma:klein-lemma-variant}, which describes how to partition multisets of small integral vectors having a given initial sum, into submultisets of predictable cardinality and uniformly determined subset sums. We remark that if we augment all vectors with an additional dummy coordinate equal to $1$, then Theorem 4.1 in~\cite{DBLP:conf/esa/CslovjecsekEPVW21} can be used to partition the multiset into submultisets of bounded cardinality. The qualitative result of \cref{lemma:klein-lemma-variant} can then be obtained by merging submultisets into multisets of uniform cardinality with \cref{lemma:general-scaled-idp} for $d=1$ as done in the proof of \cref{lemma:idp-dual-glueing}. However, a naive implementation of this would result in $N$ depending doubly exponentially on $\Delta$. In the proof of \cref{lemma:klein-lemma-variant}, we avoid this by applying a suitable normalization step at the right point in the proof by Cslovjecsek et al.~\cite{DBLP:conf/esa/CslovjecsekEPVW21}. In particular, we multiply the set of large generators of intersections of integer cones with an adaptive multiplier $\Xi/v_0^i$ in the proof of \cref{claim:decompose}. To show \cref{lemma:klein-lemma-variant}, we will use \cref{lemma:intersection-integer-cone} from~\cite{DBLP:conf/esa/CslovjecsekEPVW21} and will loosely follow the presentation from~\cite{DBLP:conf/ipco/LassotaL26} in \cref{claim:decompose}.

\begin{lemma}[Lemma 4.2 in~\cite{DBLP:conf/esa/CslovjecsekEPVW21}]
    \label{lemma:intersection-integer-cone}
    Let $d,\Delta\in\Z_{\ge0}$ be given. Then there exists positive integers $\Psi=2^{\O(d\Delta)^d}$ and $\Phi=(d\Delta)^{d^2}$ such that the following holds:  for any set of invertible matrices $\mathcal U\subseteq\{-\Delta,-\Delta+1,\dots,\Delta\}^{d\times d}$ for which $\bigcap_{U\in\mathcal U}\cone U\ne\emptyset$, there exists a generating set $P\subseteq\{-\Phi,-\Phi+1,\dots,\Phi\}^d$ such that
    \[
        \cone P=\bigcap_{U\in\mathcal U}\cone U.
    \]
    Furthermore, for any $v\in\Z^d\cap\bigcap_{U\in\mathcal U}\cone U$ it holds that
    \[
        \Psi\cdot v\in\bigcap_{U\in\mathcal U}\intcone U.
    \]
\end{lemma}
Here, $\cone\{v_1,\dots,v_\ell\}=\{\lambda_1v_1+\dots+\lambda_\ell v_\ell\ \vert\ \lambda_1,\dots,\lambda_\ell\in\R_{\ge0}\}$ and $\intcone\{v_1,\dots,v_\ell\}=\{\lambda_1v_1+\dots+\lambda_\ell v_\ell\ \vert\ \lambda_1,\dots,\lambda_\ell\in\Z_{\ge0}\}$ for a set of generating vectors $v_1,\dots,v_\ell\in\R^d$. In this context, we interpret a matrix as its set of columns.

\begin{proofof}{lemma:klein-lemma-variant}
    We augment the matrix $A$ by adding the transpose of the-all ones vector $\vecone$ to it as a new row and obtain
    \[
        B=\begin{bmatrix}
            \vecone^\top\\
            A\\
        \end{bmatrix}.
    \]
    Note that $B$ has full row-rank and $n=(2\Delta+1)^m=\O(\Delta)^m$ columns.
    
    For the sake of notation, we will index the $m+1$ rows of $B$ with $0,1,2,\dots,m$. The lemma is now equivalent to showing that given a $k\in\Z_{\ge0}$ and $c\in\{kN\}\times\Z^m$, there exist $c^1,\dots,c^k\in\{N\}\times\Z^m$ so that any $x\in\Z_{\ge0}^n$ satisfying $Bx=c$ can be decomposed into $x=x^1+\dots+x^k$ with $x^i\in\Z_{\ge0}^n$ and $Ax^i=c^i$ for $i\in[k]$. \cref{lemma:klein-lemma-variant} will straightforwardly follow from \cref{claim:decompose}.

    \begin{claimin}
        \label{claim:decompose}
        There are positive integers $N'=2^{(m\Delta)^{\O(m^2)}}$ and $N''=(m+1)\cdot(n+1)\cdot N'$ so that for any $e\in\{N'',N''+1,\dots\}\times\Z^m$, there exists a $\hat e\in\{N'\}\times\Z^m$ so that for all $x\in\Z_{\ge0}^n$ satisfying $Bx=e$, there exists an $\hat x\in\Z_{\ge0}^n$ satisfying $\hat x\le x$ and $B\hat x=\hat e$.
    \end{claimin}
    
    \begin{claimproof}
        We apply \cref{lemma:intersection-integer-cone} for dimension $d=m+1$ and coefficient size $\Delta$. Let $\Phi=(m\Delta)^{(m+1)^2}=(m\Delta)^{\O(m^2)}$ be the norm upper bound on the generating elements $P$ defined in \cref{lemma:intersection-integer-cone} and let $\Psi=2^{\O(m\Delta)^{m+1}}=2^{(m\Delta)^{\O(m)}}$ be the large multiplier defined in \cref{lemma:intersection-integer-cone}. Let $\Xi=\lcm\{1,\dots,\Phi\}=2^{(m\Delta)^{\O(m^2)}}$~\cite{Nair01021982} and choose $N'=\Psi\cdot\Xi=2^{(m\Delta)^{\O(m^2)}}$.
        
        Let $\mathcal U$ be the set of invertible matrices $U\in\{-\Delta,-\Delta+1,\dots,\Delta\}^{(m+1)\times(m+1)}$ that satisfy $e\in\cone U$. Since $e\in\bigcap_{U\in\mathcal U}\cone U$, we can apply \cref{lemma:intersection-integer-cone} to find a set of generators $P\subseteq\{-\Phi,-\Phi+1,\dots,\Phi\}^{m+1}$ generating $\bigcap_{U\in\mathcal U}\cone U$. By Caratheodory's theorem, there exists a linearly independent set $\{v^1,\dots,v^\ell\}\subseteq P$ of $\ell\le m+1$ vectors so that $e\in\cone\{v^1,\dots,v^\ell\}$. Note that we may assume that $e\in\cone B$ as the claim is trivial otherwise. Caratheodory's theorem shows that there is a $(m+1)\times(m+1)$ invertible submatrix $B'$ so that $e\in\cone B'$, which shows that $B'\in\mathcal U$. Since $\cone B'\supseteq\bigcap_{U\in\mathcal U}\cone U$ contains no nonzero vectors $v$ with $v_0\le0$, all $v^i$ must satisfy $v_0^i>0$. Define $w^i=\Psi\cdot(\Xi/v_0^i)\cdot v^i$ for $i\in[\ell]$. Since $v_0^i\in\{1,2,\dots,\Phi\}$, it holds that $v_0^i$ divides $\Xi$. In particular, $w^i/\Psi\in\Z^d\cap\bigcap_{U\in\mathcal U}\cone U$, which shows that $w^i\in\bigcap_{U\in\mathcal U}\intcone U$ by \cref{lemma:intersection-integer-cone}. Additionally, it holds that $w_0^i=\Psi\cdot\Xi=N'$. Since $e\in\cone\{v^1,\dots,v^\ell\}=\cone\{w^1,\dots,w^\ell\}$, there are nonnegative multipliers $\lambda^1,\dots,\lambda^\ell$ so that
        \[
            e=\sum_{i\in[\ell]}\lambda^i\cdot w^i.
        \]
        Since
        \[
            (m+1)\cdot(n+1)\cdot N'=N''\le e_0=\sum_{i\in[\ell]}\lambda^i\cdot w_0^i=N'\cdot\sum_{i\in[\ell]}\lambda^i,
        \]
        there must exist an $i^*\in[\ell]$ for which $\lambda^{i^*}\ge n+1$. We choose $\hat e=w^{i^*}$, which satisfies $\hat e_0=N'$.

        Now, let $x\in\Z_{\ge0}^n$ be so that $Bx=e$. Consider the polyhedron $\{x'\in\R_{\ge0}^n:Bx'=e\}$. By the Minkowski-Weyl theorem, we may write
        \[
            x=\sum_j\zeta_j\cdot u_j+\sum_j\eta_j\cdot l_j
        \]
        for vertices $u_j$ and rays $l_j$ of the polyhedron, and multipliers satisfying $\zeta_j\ge0$, $\sum_j\zeta_j=1$, $\eta_j\ge0$. By Caratheodory's theorem, we can assume that there are at most $n+1$ nonzero multipliers $\zeta_j$. In particular, there must exist a $j$ with $\zeta_j\ge\tfrac1{n+1}$. Since all terms in the decomposition of $x$ are nonnegative, it follows that $x\ge\tfrac1{n+1}\cdot u_j$. Let $U$ be a basis associated with $u_j$ and let $\pi$ be the projection on the corresponding basic variables, i.e., $\pi(u_j)=U^{-1}e\ge\veczero$. Note that by construction, $U\in\mathcal U$. Now, for $i\in[\ell]$, using that $w^i\in\intcone U$, write $w^i=U\tilde z^i$ for multipliers $\tilde z^i\in\Z_{\ge0}^{m+1}$. We estimate
        \[
            \pi(x)\ge\tfrac1{n+1}\cdot\pi(u_j)=\tfrac1{n+1}\cdot U^{-1}e=\sum_{i\in[\ell]}\tfrac{\lambda^i}{n+1}\cdot\tilde z^i.
        \]
        Nonnegativity of the terms on the right-hand side and $\lambda^{i^*}\ge n+1$ implies that $\pi(x)\ge\tilde z^{i^*}$. Since the columns of $U$ are present in $B$, we can extend $\tilde z^{i^*}$ with zeros to a solution $\hat x\in\Z_{\ge0}^n$ to $B\hat x=U\tilde z^{i^*}=w^{i^*}$, which satisfies that $\hat x\le x$. Recall that $w_0^{i^*}=N'$.
    \end{claimproof}

    Now, set $N=N''\cdot N'=2^{(m\Delta)^{\O(m^2)}}$ and let $c\in\{kN\}\times\Z^m$ for $k\in\Z_{\ge1}$ ($k=0$ is trivial). We construct vectors $\hat e^i\in\Z^{m+1}$ for $i\in[\ell]$, $\ell=(k-1)\cdot N''$, inductively as follows: apply \cref{claim:decompose} to $c-\hat e^1-\dots-\hat e^{i-1}$ to find a decomposed $\hat e^i\in\Z_{\ge0}^{m+1}$. Since $c_0=kN$ and $\hat e_0^i=N'$, it is guaranteed that $c_0-\hat e_0^1-\dots-\hat e_0^{i-1}=k\cdot N''\cdot N'-(i-1)\cdot N'\ge N''$ for $i\le\ell$. For $i\in[k-1]$ choose $c^i=\hat e^{(i-1)\cdot N''+1}+\dots+\hat e^{i\cdot N''}$ and choose $c^k=c-c^1-\dots-c^{k-1}$. These satisfy $c_0^i=N$ as required. Now, let $x\in\Z_{\ge0}^n$ be a vector such that $Bx=c$. From the guarantees of \cref{claim:decompose}, it follows that there exist solutions $y^1,\dots,y^\ell\in\Z_{\ge0}^n$ to $By^1=\hat e^1,\dots,By^\ell=\hat e^\ell$ so that $y^1+\dots+y^\ell\le x$. For $i\in[k-1]$, define $x^i=y^{(i-1)\cdot N''+1}+\dots+y^{i\cdot N''}$ and set $x^k=x-x^1-\dots-x^{k-1}$. It follows that $x=x^1+\dots+x^k$, $x^i\in\Z_{\ge0}^n$, and $Bx^i=c^i$ for $i\in[k]$ as required.
\end{proofof}

We are now ready to prove \cref{lemma:idp-primal-glueing}.

\begin{proofof}{lemma:idp-primal-glueing}
    We apply \cref{lemma:klein-lemma-variant} for dimension $m=r$, coefficient size $\Delta=U$, and the matrix $Z\in\Z^{r\times h}$ having precisely all vectors from $\{-U,-U+1,\dots,U\}^r$ as columns. This yields a positive integer $N'=2^{(rU)^{\O(r^2)}}$. We set $M=N'\cdot N$.

    Let $A\in\mathcal A$, $x\in\Z_{\ge0}^{r+s_1+\dots+s_n}$, $b\in\Z^{t_1+\dots+t_n}$, and $k\in\Z_{\ge1}$ ($k=0$ is trivial) be so that $Ax=kMb$. For notational convenience, we split the solution vector $x$ into $1+n$ bricks $v\in\Z_{\ge0}^r,x_1\in\Z_{\ge0}^{s_1},\dots,x_n\in\Z_{\ge0}^{s_n}$ so that $x=[v;x_1;\dots;x_n]$ and split the right-hand side vector into $n$ bricks $b_1\in\Z^{t_1},\dots,b_n\in\Z^{t_n}$ so that $b=[b_1;\dots;b_n]$.

    Let $\ell=kN'$. Let $\lfloor z\rfloor$ denote the coordinate-wise floor of a vector and denote $p=\lfloor-\tfrac1\ell v\rfloor+\tfrac1\ell v$. For multiplicity $k$ and target vector $\ell p\in\Z^r$, \cref{lemma:klein-lemma-variant} provides a decomposition $\ell p=q^1+\dots+q^k$ with $q^1,\dots,q^k\in\Z^r$.
    
    For each $i\in[n]$, we apply the integer decomposition property to the system $[C_i,D_i][v;x_i]=k\cdot N'\cdot N\cdot b_i=\ell Nb_i$ to decompose $[v;x_i]$ into $[v;x_i]=[w_i^1;y_i^1]+\dots+[w_i^\ell;y_i^\ell]$ with $[C_i,D_i][w_i^l;y_i^l]=Nb_i$, $w_i^l\in\Z_{\ge0}^r$, $y_i^l\in\Z_{\ge0}^{s_i}$ for $l\in[\ell]$. By applying \cref{lemma:balancing} with respect to the constraint matrix $[C_i,D_i]$, we may assume that the error vectors $e_i^l=w_i^l-\tfrac1\ell v$ are bounded by $\|e_i^l\|_\infty\le U$ for $l\in[\ell]$. Note that $\lfloor e_i^l\rfloor=e_i^l+p$ and that $\lfloor e_i^l\rfloor$ is also bounded by $U$ in the $\ell_\infty$-norm because $w_i^l$ and $U$ are both integral. For $i\in[n]$ and a rounded error vector $z\in\{-U,-U+1,\dots,U\}^r$, let $(\mu_i)_z\in\Z_{\ge0}$ be the multiplicity with which $z$ occurs in $(\lfloor e_i^l\rfloor)_{l\in[\ell]}$, i.e., $\mu_i^z=|\{l\in[\ell]:\lfloor e_i^l\rfloor=z\}|$.

    For each $i\in[n]$, we verify that $Z\mu_i=\lfloor e_i^1\rfloor+\dots+\lfloor e_i^\ell\rfloor=e_1^1+\dots+e_i^\ell+\ell p=\ell p$ and $\|\mu_i\|_1=\ell=k\cdot N'$. Thus, \cref{lemma:klein-lemma-variant} guarantees that there exists a decomposition $\mu_i=\mu_i^1+\dots+\mu_i^k$ for $\mu_i^j\in\Z_{\ge0}^h$ satisfying $\|\mu_i^j\|_1=N'$ and $Z\mu_i^j=q^j$ for $j\in[k]$. This decomposition specifies a (not necessarily unique) partition of the index set $[\ell]$ into $k$ parts $L_1,\dots,L_k\subseteq[\ell]$, each $L_j$ being of cardinality exactly $N'$ and satisfying $\sum_{l\in L_j}\lfloor e_j^l\rfloor=Z\mu_i^j=q^j$. Set $v^j=q^j-N'\lfloor-\tfrac1\ell v\rfloor$ and $x_i^j=\sum_{l\in L_j}y_i^l\in\Z_{\ge0}^{s_i}$ for $j\in[k]$. We verify that
    \[
        \sum_{l\in L_j}w_i^l=\sum_{l\in L_j}(e_i^l+\tfrac1\ell v)=\sum_{l\in L_j}(\lfloor e_i^l\rfloor-p+\tfrac1\ell v)=q^j+N'\cdot(-p+\tfrac1\ell v)=v^j
    \]
    for any $i\in[n]$, which also shows that $v^j\in\Z_{\ge0}^r$. Since $L_j$ partitions $[\ell]$, the subsolutions $[v^j;x_1^j;\dots;x_n^j]$ for $j\in[k]$ describe a decomposition of $x$, which additionally satisfies
    \begin{align*}
        C_iv^j+D_ix_i^j&=C_i\sum_{l\in L_j}w_i^l+D_i\sum_{l\in L_j}y_i^l=\sum_{l\in L_j}(C_iw_i^l+D_iy_i^l)\\
        &=\sum_{l\in L_j}Nb_i=N'\cdot Nb_i=Mb_i
    \end{align*}
    for all $i\in[n]$ and $j\in[k]$ as required.
\end{proofof}

In order to prove \cref{corollary:primal-treedepth-scaled-idp}, we will use an inductive argument on \emph{topological height}, as introduced in~\cite{DBLP:journals/mor/EisenbrandHKKLO25}. For our argument, it suffices to consider the following definition: a zero matrix has a primal decomposition of topological height $0$ and depth $0$; and a matrix $A$ has a primal decomposition of topological height $h$ and depth $d$ if and only if $A$, after permuting rows and columns, is of the form (\ref{eq:2-stage-matrix}) for $C_i\in\Z^{t_i\times r}$ and $D_i\in\Z^{t_i\times s_i}$ where all $D_i$ have a primal decomposition of topological height at most $h-1$ and depth at most $d-r$. From this, it follows that any matrix with primal treedepth at most $d$ has a primal decomposition of topological height $d$ and depth $d$, and that any matrix that has a primal decomposition of topological height $h$ and depth $d$ has primal treedepth at most $d$.

\begin{proofof}{corollary:primal-treedepth-scaled-idp}
    With induction on $h$, we will show that for each positive integral $d$ and $h$, there exists a positive integer $f(d,h)$ so that the class of matrices consisting of $A\in\{-\Delta,-\Delta+1,\dots,\Delta\}^{m\times n}$ having a primal decomposition of topological height at most $h$ and depth at most $d$ has the IDP after an $f(d,h)$-dilation.
    
    For topological height $h=1$ and any depth $d$, such matrices have at most $d$ nonzero columns. Following the arguments from \cref{lemma:general-scaled-idp}, it is seen that these matrices have the IDP after an $M$ dilation where $M=d\cdot\lcm\{1,2,\dots,\lfloor(\sqrt d\Delta)^d\rfloor\}=2^{\O((\sqrt d\Delta)^d)}$. Therefore, we may set $f(d,1)=M$.
    
    Assume that the induction hypothesis is shown for some height $h-1\ge1$ and all depths $d$, and let $A\in\{-\Delta,-\Delta+1,\dots,\Delta\}^{m\times n'}$ have a primal decomposition of height at most $h\ge2$ and depth at most $d$. Then the rows and columns of $A$ can be permuted so that it is of the form (\ref{eq:2-stage-matrix}) for $C_i\in\{-\Delta,-\Delta+1,\dots,\Delta\}^{t_i\times r}$ and $D_i\in\{-\Delta,-\Delta+1,\dots,\Delta\}^{t_i\times s_i}$ where all $D_i$ have a primal decomposition of topological height at most $h-1$ and depth at most $d-r$. Hence, $[C_i,D_i]$ has a primal decomposition of topological height at most $h-1$ and depth at most $d$. By the induction hypothesis, the matrices $[C_i,D_i]$ have the IDP after an $f(d,h-1)$-dilation. The $\ell_\infty$-Graver complexity bound by Klein and Reuter~\cite{DBLP:journals/talg/KleinR24} for multistage stochastic matrices guarantees that $g_\infty([C_i,D_i])\le U$ for $U=2^{\Delta^{2^{\O(d)}}}$. Therefore, \cref{lemma:idp-primal-glueing} guarantees that $A$ has the IDP after an $(N(r)\cdot f(d,h-1))$-dilation where $N(r)$ is the factor
    \[
        N(r)=2^{rU^{\O(r^2)}}=2^{\bigl(r\cdot2^{\Delta^{2^{\O(d)}}}\bigr)^{\O(r^2)}}
    \]
    from the lemma (or $N(0)=1$). We may choose $f(d,h)=N(1)\cdot\dots\cdot N(d)\cdot f(d,h-1)$, which guarantees that $A$ has the IDP after an $f(d,h)$-dilation independently of $r$.
    
    Crudely bounding the result of the recursion shows that the dilation
    \begin{align*}
        f(d,d)&\le(N(1)\cdot\dots\cdot N(d))^{d-1}\cdot f(d,1)\\
        &=\left(\left(2^{\bigl(d\cdot2^{\Delta^{2^{\O(d)}}}\bigr)^{\O(d^2)}}\right)^d\right)^{d-1}\cdot2^{\O((\sqrt d\Delta)^d)}=2^{2^{\Delta^{2^{\O(d)}}}}
    \end{align*}
    establishes the IDP for all bounded coefficient matrices with primal treedepth at most $d$.
\end{proofof}

%% file: sections/appendix-algorithms.tex
\section{Derivations omitted in \texorpdfstring{\cref{sec:algorithms}}{Section \getrefnumber{sec:algorithms}}}
\label{appendix:algorithms}

Here we provide the postponed proofs of \cref{theorem:dual-treedepth-algorithm,theorem:restricted-4-block-algorithm,theorem:restricted-interaction-milp}.

\begin{proofof}{theorem:dual-treedepth-algorithm}
    With a trivial preprocessing step we may remove all zero rows from $D$, which allows us to effectively assume that $m\le nd$. We replace $Bx=b^B$ by introducing $\ell r$ new variables $y_{qi}$ for $(q,i)\in[\ell]\times[r]$, enforcing that
    \begin{equation}
        \label{eq:y-aggregation}
        y_{qi}=\sum_{j\in[n]:B_{ij}=k_q}x_j
    \end{equation}
    and instead imposing constraints on $y$. For this purpose, define the $(\ell r)\times n$ matrix $Z$ by setting the coefficient in row $(q,i)\in[\ell]\times[r]$ and column $j\in[n]$ to $1$ if $B_{ij}=k_q$ and $0$ otherwise. We build the matrix
    \[
        A=\begin{bmatrix}
            Z\\
            D
        \end{bmatrix}\in\Z^{(\ell r+m)\times n}.
    \]
    In this way, $[-Iy;\veczero]+Ax=[\veczero;b^D]$ correctly encodes (\ref{eq:y-aggregation}) and $Dx=b^D$. When (\ref{eq:y-aggregation}) holds, it follows that
    \[
        b_i^B=\sum_{q\in[\ell]}k_qy_{qi}\iff b_i^B=\sum_{q\in[\ell]}k_q\Bigl(\sum_{j\in[n]:B_{ij}=k_q}x_j\Bigr)=\sum_{j\in[n]}B_{ij}x_j
    \]
    for $i\in[r]$. Hence, we can cast (\ref{eq:treedepth-ip}) as an instance of (\ref{eq:generic-2-stage-ip}) for $C=-I$, $d'=\ell r$, our composite matrix $A$, variable bounds $l,u$, $g=y\mapsto0$, the original objective $f$, right-hand side vector $b=[\veczero;b^D]$, and $2(1+\ell)r$ constraints in $Vy\le w$ given by
    \begin{alignat*}{2}
        &\sum_{q\in[\ell]}k_qy_{qi}=b_i^B,&i&\in[r],\\
        &\sum_{j\in[n]}l_j\le y_{qi}\le\sum_{j\in[n]}u_j,&\quad(q,i)&\in[\ell]\times[r].
    \end{alignat*}
    The feasible region for $y$ is contained in a ball of radius $\rho=n\|u-l\|_\infty$, yielding $\log\rho=\log(n)+L$.

    We now compute a period $M$ for which (\ref{eq:value-function-assumed-convex-extensibility}) is convex extensible by applying \cref{corollary:dual-treedepth-scaled-idp,lemma:fixed-dimension-periodic-convexity} to $d'=\ell r$ and the matrix $[A,I]$. We observe that $[A,I]$ and the further extended $A'$ in the statement of \cref{lemma:fixed-dimension-periodic-convexity} are both constraint matrices with a dual treedepth of at most $\ell r+d$ and coefficients bounded by $\Delta$. Hence, we can bound $G\le g_\infty(A')\le g_1(A')\le\O(\Delta)^{2^{\ell r+d}}$~\cite{DBLP:journals/toct/KnopPW20} and \cref{corollary:dual-treedepth-scaled-idp} shows that $[A,I]$ has the IDP after an $N$-dilation for $N=2^{\O(\Delta)^{2^{\ell r+d}}}$. Now, \cref{lemma:fixed-dimension-periodic-convexity} shows that we obtain our desired periodic convexity of (\ref{eq:value-function-assumed-convex-extensibility}) for a period
    \[
        M=\bigl(\O(\Delta)^{2^{\ell r+d}}\bigr)^{\ell r(\ell r+1)}\cdot2^{\O(\Delta)^{2^{\ell r+d}}}=2^{\O(\Delta)^{2^{\ell r+d}}}.
    \]
    
    With respect to the remaining conditions in \cref{lemma:fixed-phase-value-function-reformulation}, we can estimate $\|A\|_\infty\le n\Delta=:U$, compute a product $Ax$ in time $T_{\mathrm{mul}}=\O(n(\ell r+d))$, and solve (\ref{eq:generic-second-stage-ip}) using the dual treedepth algorithm from~\cite{DBLP:conf/ipco/HunkenschroderKLV25}, which runs in time $T_{\mathrm{opt}}(\sigma)=\Delta^{\O((\ell r+d)\cdot2^{\ell r+d})}\cdot n\log(n)\cdot\log\sigma$.
    
    Now, \cref{lemma:fixed-phase-value-function-reformulation} shows that we can solve (\ref{eq:treedepth-ip}) in time
    \begin{alignat*}{2}
        &\mathrlap{M^{\ell r}\cdot2^{\O((\ell r)^2\log(\ell r))}\cdot\Bigl(\log^{\O(1)}(\rho)+\log(\rho)}&&\\
        &&&\cdot\bigl(T_{\mathrm{opt}}((1+2U)\|u-l\|_\infty)+T_{\mathrm{mul}}+(\ell r+nd)+2(1+\ell)r\bigr)\Bigr)\\
        ={}&\bigl(&&2^{\O(\Delta)^{2^{\ell r+d}}})^{\ell r}\cdot2^{\O((\ell r)^2\log(\ell r))}\cdot\Bigl((\log(n)+L)^{\O(1)}+(\log(n)+L)\\
        &&&\cdot\bigl(\Delta^{\O((\ell r+d)\cdot2^{\ell r+d})}\cdot n\log(n)\cdot\log((1+2n\Delta)\|u-l\|_\infty)+\O(n(\ell r+d))\bigr)\Bigr)\\
        ={}&\mathrlap{2^{\O(\Delta)^{2^{\ell r+d}}}\cdot\bigl(\log^{\O(1)}(n)\cdot L^{\O(1)}\cdot\Delta^{\O((\ell r+d)\cdot2^{\ell r+d})}\cdot n\log(n)\cdot(\log(n)+\log(\Delta)+L)\bigr)}&&\\
        ={}&\mathrlap{2^{\O(\Delta)^{2^{\ell r+d}}}\cdot n\log^{\O(1)}(n)\cdot L^{\O(1)}.}&&\tag*{\qedhere}
    \end{alignat*}
\end{proofof}

\begin{proofof}{theorem:restricted-4-block-algorithm}
    We directly apply \cref{lemma:fixed-phase-value-function-reformulation} to the matrices
    \[
        C=\begin{bmatrix}
            a\\
            c_1\\
            \vdots\\
            c_n
        \end{bmatrix},
        \quad A=\begin{bmatrix}
            B_1&\cdots&B_n\\
            D_1&&\\
            &\ddots&\\
            &&D_n
        \end{bmatrix},
    \]
    $d=r+sn$, right-hand side vector $b=[b_0;b_1;\dots;b_n]$, variable bounds $l=[l_1;\dots;l_n]$, $u=[u_1;\dots;u_n]$, and encode $l_0\le x_0\le u_0$ in $Vx_0\le w$ with $\ell=2$ constraints. The objective remains $g(x_0)+f'(x_1,\dots,x_n)=f(x_0,x_1,\dots,x_n)$. \cref{corollary:n-fold-scaled-idp,proposition:idp-iff-convex-extensibility-along-line} together show that (\ref{eq:value-function-assumed-convex-extensibility}) is convex extensible on a $1$-dimensional subspace for $M=2^{\O((r+s)\Delta)^{r(r+s)}}$. With respect to the remaining conditions in \cref{lemma:fixed-phase-value-function-reformulation}, we can estimate $\|A\|_\infty\le tn\Delta=:U$, compute a product $Ax$ in time $T_{\mathrm{mul}}=\O(n(r+s)t)$, and solve (\ref{eq:generic-second-stage-ip}) using the dual treedepth algorithm from~\cite{DBLP:conf/ipco/HunkenschroderKLV25}, which runs in time $T_{\mathrm{opt}}(\sigma)=(rs\Delta)^{\O(rs(r+s))}\cdot tn\log(tn)\cdot\log\sigma$. We again have $\log\rho\le L$.

    \cref{lemma:fixed-phase-value-function-reformulation} shows that we can solve (\ref{eq:n-fold-ip}) in time
    \begin{alignat*}{2}
        &M\cdot{}&&\Bigl(\log^{\O(1)}(\rho)+\log(\rho)\\
        &&&\cdot\bigl(T_{\mathrm{opt}}((1+2U)\|u-l\|_\infty)+T_{\mathrm{mul}}+(r+sn)+2\bigr)\Bigr)\\
        ={}&\mathrlap{2^{\O((r+s)\Delta)^{r(r+s)}}\cdot\Bigl(L^{\O(1)}+L\cdot\bigl((rs\Delta)^{\O(rs(r+s))}}&&\\
        &&&\cdot tn\log(tn)\cdot\log((1+2tn\Delta)\|u-l\|_\infty)+\O(n(r+s)t)+r+sn+2\bigr)\Bigr)\\
        ={}&\mathrlap{2^{\O((r+s)\Delta)^{r(r+s)}}\cdot\Bigl(L^{\O(1)}\cdot\bigl((rs\Delta)^{\O(rs(r+s))}\cdot tn\log(tn)\cdot(\log(tn\Delta)+L)\bigr)\Bigr)}&&\\
        ={}&\mathrlap{2^{\O((r+s)\Delta)^{r(r+s)}}\cdot r\cdot tn\log^{\O(1)}(tn)\cdot L^{\O(1)}.}&&\tag*{\qedhere}
    \end{alignat*}
\end{proofof}

\begin{proofof}{theorem:restricted-interaction-milp}
    Using standard proximity arguments~\cite{DBLP:journals/mp/CookGST86}, we may assume that all variable domains are bounded, i.e., it suffices to solve
    \begin{equation}
        \label{eq:milp-bounded}
        \min\bigl\{c_\R^\top x_\R+c_\Z^\top x_\Z\bigm\vert A_\R x_\R+A_\Z x_\Z=b,\,\veczero\le x\le u_\R,\,\veczero\le x_\Z\le u_\Z,\,x_\R\in\R^{n_\R},\,x_\Z\in\Z^{n_\Z}\bigr\}
    \end{equation}
    for some $u_\R\in\R^{n_\R}$ and $u_\Z\in\Z^{n_\Z}$. By permuting the rows of $[A_\R,A_\Z,b]$, we may write
    \[
        [A_\R,A_\Z,b]=\begin{bmatrix}
            X_\R&&b_\R\\
            Y_\R&Y_\Z&b^*\\
            &X_\Z&b_\Z
        \end{bmatrix}
    \]
    for $X_\R\in\R^{m_\R\times n_\R}$, $Y_\R\in\R^{k\times n_\R}$, $X_\Z\in\R^{m_\Z\times n_\Z}$, $Y_\Z\in\R^{k\times n_\Z}$, $b_\R\in\R^{m_\R}$, $b^*\in\R^k$, and $b_\Z\in\Z^{m_\Z}$.

    We create $k$ variables $y\in\Z^k$ that represent the contribution $Y_\R x_\R$ of the fractional variables to the shared constraints. Since $\veczero\le x_\R\le u_\R$, we may impose the constraint $\|Y_\R x_\R\|_\infty\le u^*$ in (\ref{eq:milp-bounded}) for some $u^*$ with polynomial encoding length.

    Later, we will define a finite objective $g\colon\R^k\to\R$ which models the optimal objective contribution $g'\colon\R^k\to\R\cup\{\infty\}$ of the fractional variables under the constraint $Y_\R x_\R=y$, which is given by
    \begin{equation}
        \label{eq:milp-continuous-primal}
        g'(y)=\min\{c_\R^\top x_\R\ \vert\ X_\R x=b_\R,\,Y_\R x_\R=y,\,\veczero\le x_\R\le u_\R,\,x_\R\in\R^{n_\R}\}.
    \end{equation}
    First, we focus on the integral variables $x_\Z$. To solve (\ref{eq:milp-bounded}) we will apply \cref{lemma:fixed-phase-value-function-reformulation} to the constraint matrices
    \[
        C=I,\quad A=\begin{bmatrix}
            Y_\Z\\
            X_\Z
        \end{bmatrix}
    \]
    with dimension $d'=k$, right-hand side vector $b=[b^*;b_\Z]$, variable bounds $l=\veczero$, $u'=u_\Z$, and objective $f(x_\Z)=c_\Z^\top x_\Z$. We encode $-u^*\le y_i\le u^*$ for all $i\in[k]$ into $Vy\le w$, which has a diameter $\rho=\O(\log u^*)$, of polynomially bounded encoding length.

    The constraint matrix $[A,I]$ has primal or dual treedepth at most $d+1$. Therefore, we have that $[A,I]$ has the IDP after an $N$-dilation for $N$ being the product of the multipliers from \cref{corollary:dual-treedepth-scaled-idp,corollary:primal-treedepth-scaled-idp}, which is a function of $\Delta$ and $d$. Similarly, adding another (submatrix of an) identity matrix increases the treedepth by at most $1$, which allows us to apply \cref{lemma:fixed-dimension-periodic-convexity} for $d'=k$ and $U$ being the maximum of the Graver complexity bounds $\O(\Delta)^{2^{d+2}}$~\cite{DBLP:journals/toct/KnopPW20} and $2^{\Delta^{2^{\O(d+2)}}}$~\cite{DBLP:journals/talg/KleinR24} for matrices with bounded dual or primal treedepth. This yields an $M$, depending only on $\Delta$, $d$ and $k$, which guarantees the convex extensibility of (\ref{eq:value-function-assumed-convex-extensibility}). The pure integer problem (\ref{eq:generic-second-stage-ip}) can be solved with one of the bounded treedepth algorithms from~\cite{DBLP:conf/ipco/HunkenschroderKLV25}, which runs in time $f'(\Delta,d,k)$ times a polynomial in the encoding length of the instance. We can use trivial bounds for $\|A_\Z\|_\infty\le\Delta n_\Z$ and $T_{\mathrm{mul}}=(k+m_\Z+n_\Z)^{\O(1)}$.

    We now construct a convex and finite objective function $g$ that models the feasibility and objective contribution of $g'(Y_\R x_\R)$ in (\ref{eq:milp-bounded}). A dual linear program of (\ref{eq:milp-continuous-primal}) is
    \[
        \max\{v(y)^\top z\ \vert\ z\in P\}
    \]
    for $v(y)=[b_\R;y;u_\R]$ and the polyhedron $P=\{[z_X;z_Y;z_u]\in\R^{m_\R}\times\R^k\times\R_{\le0}^{n_\R}:X_\R^\top z_X+Y_\R^\top z_Y+z_u\le c_\R\}$. Since (\ref{eq:milp-continuous-primal}) is either finite or infeasible, $P=\emptyset$ implies that (\ref{eq:milp-continuous-primal}) is infeasible for all $y$, which in turn implies that (\ref{eq:milp-bounded}) is infeasible. We can check whether $P$ is empty in polynomial time, which shows that it suffices to focus on the case where $P\ne\emptyset$. In polynomial time we compute a bound $\alpha\ge0$ (with polynomial encoding length) so that any minimal face of $P$ contains a point with $\ell_\infty$-norm bounded by $\alpha$. Set
    \begin{equation}
        \label{eq:milp-finite-continuous-primal}
        \begin{aligned}
            g(y)=\max\bigl\{&\max\{v(y)^\top z\ \vert\ \|z\|_\infty\le\alpha,\,z\in P\},\\
            &\beta\cdot\max\{v(y)^\top z\ \vert\ \|z\|_\infty\le1,\,z\in\rec(P)\}\bigr\}
        \end{aligned}
    \end{equation}
    for some sufficiently large $\beta>0$. Here, $\rec(P)$ denotes the recession cone of $P$. Note that $g$ is finite. Duality shows that if $g'(y)<\infty$, then $g(y)=g'(y)$. On the other hand, there exists a $\delta>0$ so that if $y\in\Z^k$ and $g'(y)=\infty$, then $\max\{v(y)^\top z\ \vert\ \|z\|_\infty\le1,\,z\in\rec(P)\}\ge\delta$. Therefore, when $\beta$ is sufficiently large, if $y\in\Z^k$ and $g'(y)=\infty$, the value of $g(y)$ plus $c_\Z^\top x_\Z$ for any $\veczero\le x_\Z\le u_\Z$ will be strictly larger than the value of $g(y')$ plus $c_\Z^\top x_\Z'$ for any $y'$ for which $g'(y')<\infty$ and $\veczero\le x_\Z'\le u_\Z$. In this way, (\ref{eq:milp-bounded}) is feasible if and only if the optimal solution $(y,x_\Z)$ to (\ref{eq:generic-2-stage-ip}) produced by \cref{lemma:fixed-phase-value-function-reformulation} satisfies $g'(y)<\infty$, in which case it can be completed to an optimal solution $(x_\R,x_\Z)$ by solving (\ref{eq:milp-continuous-primal}) for this $y$. Again, the value of $\beta$ need not be computed since it can be simulated using lexicography. Hence, a comparison oracle for $g$ can be implemented by solving both linear programs in (\ref{eq:milp-finite-continuous-primal}) in polynomial time.
    
    Combining the polynomiality of the comparison oracle of $g,f$ with the number of steps of the algorithm from \cref{lemma:fixed-phase-value-function-reformulation}, which is bounded by $f''(\Delta,d,k)\cdot|I|^{\O(1)}$ for some computable function $f'$, shows that solving (\ref{eq:milp}) is FPT when parameterized by $\Delta,d,k$.
\end{proofof}